\documentclass[reqno, 11pt]{amsart}
\usepackage{etex}
\usepackage[utf8]{inputenc}
\usepackage[english]{babel}
\usepackage{amssymb,amsmath,amsthm}
\usepackage{graphicx,geometry,caption}
\usepackage{float}
\usepackage{enumerate}
\usepackage{tikz}
\usetikzlibrary{patterns}
\usepackage{pgfplots}
\usepackage[all]{xy}
\usepackage{tikz-cd}
\usepackage{comment}
\usepackage{cite}
\usepackage{mathabx}
\usepackage{setspace}
\usepackage{MnSymbol}
\pgfplotsset{compat=1.10}

\numberwithin{equation}{section}
\newcommand{\extp}{\@ifnextchar^\@extp{\@extp^{\,}}}
\def\extp^#1{\mathop{\bigwedge\nolimits^{\!#1}}}
\makeatother

\theoremstyle{plain}
\newtheorem{teo}{Theorem}[section]
\newtheorem{prop}{Proposition}[section]
\newtheorem{cor}{Corollary}[section]
\newtheorem{lema}{Lemma}[section]

\theoremstyle{definition}
\newtheorem{defi}{Definition}[section]
\newtheorem{rem}{Remark}[section]

\begin{document}

\title{Automorphisms of descending mod-p central series}

\author{Ricard Riba}
\address{Universitat Autònoma de Barcelona, Departament de Matemàtiques, Bellaterra, Spain}
\email{riba@mat.uab.cat}
\thanks{This work was partially supported by MEC grant MTM2016-80439-P}



\date{\today}

\begin{abstract}
Given a free group $\Gamma$ of finite rank $n$ and a prime number $p,$ denote by  $\Gamma_k^\bullet$ the $k^\text{th}$ layer of the Stallings ($\bullet=S$) or Zassenhaus ($\bullet=Z$) $p$-central series, by
$\mathcal{N}_{k}^\bullet$ the quotient $\Gamma/\Gamma_{k+1}^\bullet$ and by $\mathcal{L}_{k}^\bullet$ the quotient $\Gamma_k^\bullet /\Gamma_{k+1}^\bullet.$
In this paper we prove that there is a non-central extension of groups
$
0 \longrightarrow Hom(\mathcal{N}^\bullet_1, \mathcal{L}^\bullet_{k+1}) \longrightarrow Aut\;\mathcal{N}^\bullet_{k+1} \longrightarrow  Aut \;\mathcal{N}^\bullet_k \longrightarrow 1,
$
which splits if and only if $k=1$ and $p$ is odd if $\bullet=Z$  or, $k=1$ and $(p,n)= (3,2), (2,2)$ if $\bullet=S$.
Moreover, if we denote by $IA^p(\mathcal{N}^\bullet_k )$ the subgroup of $Aut \;\mathcal{N}^\bullet_k$ formed by the automorphisms that acts trivially on $\mathcal{N}_1^\bullet,$ then the restriction of this extension to $IA^p(\mathcal{N}^\bullet_{k+1})$ give us a non-split central extension of groups
$
0 \longrightarrow Hom(\mathcal{N}^\bullet_1,\mathcal{L}^\bullet_{k+1}) \longrightarrow IA^p(\mathcal{N}^\bullet_{k+1}) \longrightarrow IA^p(\mathcal{N}^\bullet_k ) \longrightarrow 1.
$
\end{abstract}

\maketitle

\section{Introduction}
\label{sec_intro}
Given a free group of finite rank $\Gamma,$ denote by $\Gamma_k$ the $k^\text{th}$ layer of its lower central series, defined inductively by $\Gamma_1=\Gamma,$ $\Gamma_{k+1}=[\Gamma,\Gamma_k]$, by $\mathcal{N}_k$ the quotient $\Gamma /\Gamma_{k+1},$ and by $\mathcal{L}_k$ the quotient $\Gamma_k /\Gamma_{k+1}.$

In 1963 S. Andreadakis studied the automorphisms of free groups and free nilpotent groups. In particular he proved that for every $k\in \mathbb{N},$ the homomorphism $ Aut\; \mathcal{N}_{k+1}\rightarrow Aut \;\mathcal{N}_{k},$ induced by modding out the subgroup  $\mathcal{L}_{k+1} \subset \mathcal{N}_{k+1} $ is surjective (see \cite{andrea}).
Later, in 1993, S. Morita characterized the kernel of this epimorphism giving a non-central extension (see \cite{mor_ext}):
\begin{equation}
\label{ext_aut_intro}
\xymatrix@C=7mm@R=10mm{0 \ar@{->}[r] & Hom(\mathcal{N}_1, \mathcal{L}_{k+1}) \ar@{->}[r] & Aut\;\mathcal{N}_{k+1} \ar@{->}[r] & Aut \;\mathcal{N}_k \ar@{->}[r] & 1.}
\end{equation}
Moreover he proved that the restriction of this extension to $IA(\mathcal{N}_{k+1})$, the subgroup of $Aut\;\mathcal{N}_{k+1}$ formed by the elements that act trivially on $\mathcal{N}_1,$ becomes a central extension.
Then in 2001 W. Pitsch gave a functorial construction of extension \eqref{ext_aut_intro} and proved that it does not split for all $k\in \mathbb{N}$ (see \cite{pitsch2}).

While trying to generalize the aforementioned  work of S. Morita to $\mathbb{Z}/p$  we became interested in the analogous extensions but with the  mod $p$ Zassenhauss and Stallings central series instead of the lower central series.

This article is focused on the construction of these extensions, the study of their centrality and the existence of a splitting of these extensions.
To be more precise, analogously to the lower central series case, given a free group $\Gamma$ of finite rank $n$ and a prime number $p,$ denote by  $\Gamma_k^\bullet$ the $k$-layer of the Stallings ($\bullet=S$) or Zassenhaus ($\bullet=Z$) $p$-central series, by
$\mathcal{N}_{k}^\bullet$ the quotient $\Gamma/\Gamma_{k+1}^\bullet$ and by $\mathcal{L}_{k}^\bullet$ the quotient $\Gamma_k^\bullet /\Gamma_{k+1}^\bullet.$
The aim of this article is to prove the following result:

\begin{teo}
There is a non-central extension of groups
\begin{equation*}
\xymatrix@C=7mm@R=10mm{0 \ar@{->}[r] & Hom(\mathcal{N}^\bullet_1, \mathcal{L}^\bullet_{k+1}) \ar@{->}[r]^-{i} & Aut\;\mathcal{N}^\bullet_{k+1} \ar@{->}[r]^-{\psi_k^\bullet} & Aut \;\mathcal{N}^\bullet_k \ar@{->}[r] & 1,}
\end{equation*}
which splits if and only if $k=1$ and $p$ is odd if $\bullet=Z$  or, $k=1$ and $(p,n)= (3,2), (2,2)$ if $\bullet=S$.
Moreover, if we denote by $IA^p(\mathcal{N}^\bullet_k )$ the subgroup of $Aut \;\mathcal{N}^\bullet_k$ formed by the automorphisms that acts trivially on $\mathcal{N}_1^\bullet,$ then the restriction of this extension to $IA^p(\mathcal{N}^\bullet_{k+1})$ give us a non-split central extension of groups
\begin{equation*}
\xymatrix@C=7mm@R=10mm{0 \ar@{->}[r] & Hom(\mathcal{N}^\bullet_1,\mathcal{L}^\bullet_{k+1}) \ar@{->}[r]^-{i} & IA^p(\mathcal{N}^\bullet_{k+1}) \ar@{->}[r]^-{\psi^\bullet_k} & IA^p(\mathcal{N}^\bullet_k ) \ar@{->}[r] & 1. }
\end{equation*}
\end{teo}

The plan of this work is the following:

In Section~\ref{sec_prelim}, we give some background about commutator calculus, $p$-central series and $p$-coverings.
In Section~\ref{sec_auto}, we give the functorial construction of the extensions  and we study their centrality.
Finally, in Section~\ref{sec_split}, we discuss the existence of a splitting of these extensions.

Throughout this paper we let $G$ denote an arbitrary group and $p$ an arbitrary prime number. 

\section{Preliminaries}
\label{sec_prelim}
In this Section we assemble classical results, maybe in a not so classical presentation in order to keep the whole work reasonably self-contained. Experts may safely skip this Section.
\subsection{Commutator calculus}
\label{sec_comm}
Given two elements $x,y$ of a fixed group $G,$ we denote by $[x,y]=xyx^{-1}y^{-1}$ their commutator and by $x^y=yxy^{-1}$ the conjugate of $x$ by $y.$ The commutator and the conjugate are related by the following classical identities (c.f. Chapter 10, \cite{hall4}):
\begin{enumerate}
\item $x^y=[y,x]x,$
\item $[y,x]=[x,y]^{-1},$
\item $[x,y^{-1}]=[y,x]^{y^{-1}} \text{ and } [x^{-1},y]= [y,x]^{x^{-1}},$
\item $[xy,z]=[y,z]^x[x,z] \text{ and } [x,yz]=[x,y][x,z]^y,$
\item $[x^y,[y,z]][y^z,[z,x]][z^x,[y,x]]=1 \text{ (Hall-Witt identity)}.$
\end{enumerate}
There is a more general version of these identities, due to
P. Hall in \cite{hall2}.
Usually when doing commutator calculus (c.f. Chapter 5, \cite{magnus1}) one atributes to generators a weigh $1,$ to commutator a weigh $2$ and so on. By introducing the notion of \textit{complex commutator with weights,} P. Hall generalized this by allowing weight elements in a group in a much more flexible way.

\begin{defi}[complex commutator, \cite{hall2}]
Let $P_1,P_2,\ldots, P_r$ be any $r$ elements of a group $G.$ We shall define by induction what we mean by a \textit{complex commutator of weight $w$ in the components $P_1,P_2,\ldots, P_r.$} The complex commutators of weight $1$ are the elements $P_1,P_2,\ldots, P_r$ themselves.
Supposing the complex commutators of all weights less than $w$ have already been defined, then those of weight $w$ consist of all the expressions of the form $[S,T],$ where $S$ and $T$ are any complex commutators of weights $w_1$ and $w_2$ in the components $P_1,P_2,\ldots P_r$ respectively, such that $w_1+w_2=w.$

The weight of a complex commutator is, of course, always relative to a choice of components; these must be specified before weight can be determined.
For example, $[[P,Q],R]$ is of weight $3$ in the three components $P,$ $Q$ and $R;$ but of weight $2$ in the two components $[P,Q]$ and $R.$

\end{defi}

\begin{teo}[P. Hall, Theorem 3.2. in \cite{hall2}]
\label{teo_hall}
If $p$ is a prime, $\alpha$ is a positive integer, $P$ and $Q$ are any two elements of a group $G,$ and
$$R_1,R_2,\ldots ,R_i,\ldots \qquad (R_1=P,\quad R_2=Q)$$
are the various formally distinct complex commutators of $P$ and $Q$ arranged in order of increasing weights, then integers $n_1,n_2,\ldots ,n_i,\ldots$ can be found $(n_1=n_2=p^\alpha)$ such that
$$(PQ)^{p^{\alpha}}=R_1^{n_1}R_2^{n_2}\cdots R_i^{n_i}\cdots$$
and if the weight $w_i$ of $R_i$ in $P$ and $Q$ satisfies the inequality $p^{\beta-1} \leq w_i<p^\beta \leq p^\alpha,$ then $n_i$ is divisible by $p^{\alpha-\beta +1}.$
\end{teo}

\subsection{On $p$-central series}
We come now to the basic objects of our study. In full generality
a \textit{$p$-central series of $G$} is a sequence of subgroups
$ \{1\} \triangleleft A_n \triangleleft \ldots \triangleleft A_2
\triangleleft A_1 =G$
such that each quotient $A_i /A_{i+1}$ is central in $G/A_{i+1}$
and $p$-elementary abelian.

We will focus on two universal $p$-central series:
\begin{itemize}
\item The Zassenhaus mod-$p$ central series $\{G_k^Z\}$,
which is the fastest descending series satisfying that
$[G^Z_k,G^Z_l]< G^Z_{k+l},\quad (G_k^Z)^p < G^Z_{pk}.$
\item The Stallings mod-$p$ central series $\{G^S_k\}$ (also known as the lower $p$-central series), which is the fastest descending series satisfying
$[G_k^S, G^S_l]< G^S_{k+l},\quad (G^S_k)^p < G^S_{k+1}.$
\end{itemize}

The different layers of these series have a concrete description as follows:
  
\begin{defi}[Zassenhaus]
Given a group $G,$ the lower central series $\{G_k\}_k$ of $G$ is defined inductively by
$G_1=G$ and $G_k=[G,G_{k-1}].$
Then the Zassenhaus mod-$p$ central series $\{G_k^Z\}$ of $G$ is defined by the rule:
\begin{equation}
\label{zassenhaus_series_def}
G_k^Z=\prod_{ip^j\geq k} (G_i)^{p^j}.
\end{equation}
\end{defi}

\begin{defi}[Stallings]
The Stallings mod-$p$ central series $\{G^S_k\}$ of $G$ is defined inductively by:
$$G^S_1=G \quad \text{ and }\quad G^S_k=[G,G_{k-1}^S](G_{k-1}^S)^p.$$
In \cite[page 242]{hup2} the author proves that the Stallings series has a description analogous to the Zassenhaus series, namely:
\begin{equation}
\label{stallings_series_alternative_def}
G_k^S=\prod_{i+j= k} (G_i)^{p^j}.
\end{equation}
\end{defi}

Notice that by the universal properties of these series we have the inclusions $G^S_k < G^Z_k$ for all $k\in \mathbb{N}.$

We now specialize these definitions for a free group $\Gamma$ of finite rank. Set:
$$\mathcal{N}^\bullet_k=\Gamma/\Gamma^\bullet_{k+1},\qquad \mathcal{L}^\bullet_k=\Gamma^\bullet_k/\Gamma^\bullet_{k+1},$$
$$\widetilde{\mathcal{N}_{k}^\bullet}=\frac{\Gamma}{[\Gamma,\Gamma^\bullet_{k}](\Gamma^\bullet_{k})^p}, \qquad \widetilde{\mathcal{L}_{k}^\bullet}=\frac{\Gamma^\bullet_{k}}{[\Gamma,\Gamma^\bullet_{k}](\Gamma^\bullet_{k})^p},$$
where $\bullet=S \text{ or } Z.$

Feeding back these quotients into the construction of the both central series and computing as in Theorem 5.3 of \cite{magnus1} we get

\begin{prop}
\label{prop_ident_L_versal}
For every $k\geq 1,$ there are isomorphisms
$$
\mathcal{L}_k^S\cong(\mathcal{N}_k^S)^S_k, \quad \mathcal{L}_k^Z\cong(\mathcal{N}_k^Z)^Z_k, \quad \widetilde{\mathcal{L}_k^Z}\cong(\widetilde{\mathcal{N}_k^Z})^Z_k, \quad\frac{\Gamma^Z_{k+1}}{[\Gamma, \Gamma^Z_{k}](\Gamma^Z_{k})^p}\cong(\widetilde{\mathcal{N}_{k}^Z})^Z_{k+1}.
$$
\end{prop}

By construction, the groups $\mathcal{N}^\bullet_k,$ $\mathcal{L}^\bullet_k,$ $\widetilde{\mathcal{N}_{k}^\bullet},$ $\widetilde{\mathcal{L}_{k}^\bullet}$ fit into a push-out diagram of central group extensions:
\begin{equation}
\label{p-central_ext_push-out}
\xymatrix@C=7mm@R=10mm{
0 \ar@{->}[r] &\widetilde{\mathcal{L}_{k+1}^\bullet} \ar@{->}[r] \ar@{->}[d] & \widetilde{\mathcal{N}_{k+1}^\bullet} \ar@{->}[r] \ar@{->}[d] & \mathcal{N}^\bullet_k \ar@{->}[r] \ar@{=}[d] & 1 \\
0 \ar@{->}[r] & \mathcal{L}_{k+1}^\bullet \ar@{->}[r] &\mathcal{N}_{k+1}^\bullet \ar@{->}[r] & \mathcal{N}^\bullet_k \ar@{->}[r] & 1.}
\end{equation}
Centrality of these extensions is checked by direct computations using Hall identities.

\begin{rem}
Notice that these two central extensions coincide for the Stallings case and differ for the Zassenhaus case.
\end{rem}

\subsection{On $p$-coverings}
It turns out, as we will see, that the extensions in diagram \eqref{p-central_ext_push-out} are not arbitrary, they are $p$-coverings.
Recall that if $G$ is a group, then its \textit{ Frattini subgroup $\mathcal{F}(G)$} is the intersection of all maximal proper subgroups of $G.$
The following definition is inspired on Section 9.5 of \cite{leed}.

\begin{defi}
A \textit{$p$-covering of $G$} is a central extension of the form
$$\xymatrix@C=7mm@R=10mm{ 0 \ar@{->}[r] &  A \ar@{->}[r] & P \ar@{->}[r] & G \ar@{->}[r] & 1},$$
where $A$ is a $\mathbb{Z}/p$-vector space and $A\leq \mathcal{F}(P).$
In this case we say that $P$ is a $p$-covering group of $G.$
We will say that this extension is \textit{the universal $p$-covering of $G$} and $P$ is \textit{the universal $p$-covering group of $G$} (up to canonical isomorphisms),
if for any other $p$-covering
$\xymatrix@C=7mm@R=10mm{ 0 \ar@{->}[r] &  A' \ar@{->}[r] & E \ar@{->}[r] & G \ar@{->}[r] & 0},$
there is a push-out diagram
$$\xymatrix@C=7mm@R=10mm{ 0 \ar@{->}[r] &  A \ar@{->}[r] \ar@{->}[d] & P \ar@{->}[d] \ar@{->}[r] & G \ar@{->}[r] \ar@{=}[d] & 1 \\
0 \ar@{->}[r] &  A' \ar@{->}[r] & E \ar@{->}[r] & G \ar@{->}[r] & 1.}$$
\end{defi}

\textbf{If $G$ is a finite $p$-group,} by Proposition 1.2.4 in \cite{leed}, the Frattini subgroup of $G$ is
$\mathcal{F}(G)=[G,G]G^p$ and the rank of the quotient $G/\mathcal{F}(G)$ coincides with $d(G),$ the minimal cardinality of a generating set of $G.$

In this case, given $\xymatrix@C=5mm@R5mm{
0 \ar@{->}[r] & R \ar@{->}[r] & F \ar@{->}[r] & G \ar@{->}[r] & 1}$ a presentation of $G$ with $d(F)=d(G),$
we have that $R\leq [F,F]F^p.$ Moreover applying the mod $p$ Hopf formula, given in \cite{graham}, to this presentation of $G$ we obtain that $H_2(G;\mathbb{Z}/p)\cong R/[R,F]R^p.$ 

Then the quotient of this presentation by $[F,R]R^p$ give us a $p$-covering of $G$
$$
\xymatrix@C=7mm@R=10mm{
0 \ar@{->}[r] & H_2(G;\mathbb{Z}/p) \ar@{->}[r] & \frac{F}{[F,R]R^p} \ar@{->}[r] & G \ar@{->}[r] & 1. }
$$
By Proposition 9.5.13 in \cite{leed}, this is in fact the universal $p$-covering of $G.$

Back to our object of interest, observe that the groups $\mathcal{N}_k^\bullet$ are $p$-groups. Moreover:

\begin{prop}
\label{prop_univ_p-cov_stall_zass}
The central extension $\xymatrix@C=5mm@R=5mm{
0 \ar@{->}[r] & H_2(\mathcal{N}_k^\bullet;\mathbb{Z}/p) \ar@{->}[r] &\widetilde{\mathcal{N}_{k+1}^\bullet} \ar@{->}[r] & \mathcal{N}^\bullet_k \ar@{->}[r] & 1 }$ is the universal $p$-covering of $\mathcal{N}_{k}^\bullet.$
\end{prop}

which is a direct consequence of the following computation

\begin{lema}
\label{lema_Hopf_iso_modp}
For any prime number $p,$
$$H_2(\mathcal{N}^\bullet_{k};\mathbb{Z}/p)\cong \frac{\Gamma^\bullet_{k+1}}{[\Gamma,\Gamma^\bullet_{k+1}](\Gamma^\bullet_{k+1})^p}=\widetilde{\mathcal{L}_{k+1}^\bullet}.$$
\end{lema}

\begin{proof}
Applying the mod $p$ Hopf formula, given in \cite{graham}, to the presentation of $\mathcal{N}_k^\bullet$ given by
$1\rightarrow \Gamma^\bullet_{k+1}\rightarrow \Gamma \rightarrow \mathcal{N}^\bullet_{k}\rightarrow 1,$
since $\Gamma^\bullet_{k+1} \subset [\Gamma,\Gamma]\Gamma^p,$ we get the result.
\end{proof}

\section{Automorphisms of $p$-coverings}
\label{sec_auto}
We now exhibit some results about automorphisms of the universal $p$-covering of an arbitrary $p$-group $G.$ The main property of the universal $p$-covering of $G$ that we will use here, is the fact that, under the canonical action of $G$ on $H^2(G;H_2(G;\mathbb{Z}/p)),$
the cohomology class of the universal $p$-covering of $G$ is stable.

Notice that applying the Universal coefficients theorem to the group $G$ and the $\mathbb{Z}/p$-module $H_2(G;\mathbb{Z}/p)$ with trivial $G$-action, we have an induced natural isomorphism:
\begin{equation}
\label{iso_uct_h2}
\eta: \xymatrix@C=7mm@R=10mm{ H^2(G;H_2(G;\mathbb{Z}/p))
\ar@{->}[r]^-{\sim} & Hom_{\mathbb{Z}/p}(H_2(G;\mathbb{Z}/p),H_2(G;\mathbb{Z}/p)).}
\end{equation}

\begin{prop}
\label{prop_versal_ext_general}
Given a presentation of $G,$ $\xymatrix@C=7mm@R=10mm{
0 \ar@{->}[r] & R \ar@{->}[r] & F \ar@{->}[r] & G \ar@{->}[r] & 1}.$ If we denote $v_p\in H^2(G;H_2(G;\mathbb{Z}/p))$ the cohomology class associated to the universal $p$-covering of $G,$
\begin{equation}
\label{versal_ext_general}
\xymatrix@C=7mm@R=7mm{
0  \ar@{->}[r] & H_2(G;\mathbb{Z}/p) \ar@{->}[r]
 & \frac{F}{[F,R]R^p} \ar@{->}[r] & G \ar@{->}[r] & 1. }
\end{equation}
Then $\eta(v_p)=id.$
\end{prop}

\begin{proof}
Arguing as in Chapter 5 of \cite{brown} and also as in Chapter 9 of \cite{leed}, if we start with a presentation of $G,$ $\xymatrix@C=5mm@R=5mm{
0 \ar@{->}[r] & R \ar@{->}[r] & F \ar@{->}[r] & G \ar@{->}[r] & 1}$
with $d(F)=d(G),$ the central extension of groups
$\xymatrix@C=5mm@R=5mm{
0 \ar@{->}[r] & \frac{R}{[F,R]R^p} \ar@{->}[r] & \frac{F}{[F,R]R^p} \ar@{->}[r] & G \ar@{->}[r] & 1 }$
induces the mod $p$ Hopf isomorphism $H_2(G;\mathbb{Z}/p)\xrightarrow{\sim} \frac{R}{[F,R]R^p}.$
As a consequence, by the naturality of $\eta,$ we get an extension of groups
$\xymatrix@C=5mm@R=5mm{
0 \ar@{->}[r] & H_2(G;\mathbb{Z}/p) \ar@{->}[r] & \frac{F}{[F,R]R^p} \ar@{->}[r] & G \ar@{->}[r] & 1, }$
witch induces the identity in $H_2(G;\mathbb{Z}/p).$
\end{proof}

\begin{lema}
\label{lema_eq_cocy}
Let $v_p\in H^2(G;H_2(G;\mathbb{Z}/p))$ denote the cohomology class given by the preimage of the identity by the isomorphism $\eta$ of \eqref{iso_uct_h2}. Then for each $\phi\in Aut(G),$ the following equality holds $(H_2(\phi; \mathbb{Z}/p))_*(v_p)=\phi^*(v_p).$

Here $H_2(\phi; \mathbb{Z}/p)$ denotes the element of $Aut(H_2(\phi; \mathbb{Z}/p))$ induced by $\phi.$
\end{lema}

\begin{proof}
Fix an element $\phi\in Aut(G),$ by construction of the isomorphism $\eta$ and its naturality,
we have that
\begin{align*}
\eta(\phi^*(v_p))= & (H_2(\phi; \mathbb{Z}/p))^*(id)= H_2(\phi; \mathbb{Z}/p), \\
\eta((H_2(\phi; \mathbb{Z}/p))_*(v_p))= & (H_2(\phi; \mathbb{Z}/p))_*(id)=H_2(\phi; \mathbb{Z}/p).
\end{align*}
Therefore $\phi^*(v_p)=(H_2(\phi; \mathbb{Z}/p))_*(v_p),$ as desired.
\end{proof}

In particular an automorphism of $G$ lifts to an automorphism of $E:$

\begin{cor}
\label{cor_ver_ext_mod_p}
Let $G$ be a group and
$\xymatrix@C=5mm@R=5mm{
0  \ar@{->}[r] & H_2(G;\mathbb{Z}/p) \ar@{->}[r]
 & E \ar@{->}[r] & G \ar@{->}[r] & 1 }$
 the universal $p$ covering of $G.$
For every element $\phi\in Aut(G)$ there exists an element $\Phi\in Aut(E)$ such that the following diagram commutes
\begin{equation}
\xymatrix@C=10mm@R=10mm{
 0 \ar@{->}[r] & H_2(G;\mathbb{Z}/p) \ar@{->}[r] \ar@{->}[d]^-{H_2(\phi; \mathbb{Z}/p)} & E \ar@{->}[r] \ar@{->}[d]^-{\Phi} & G \ar@{->}[r] \ar@{->}[d]^-{\phi} & 1\\
0  \ar@{->}[r] & H_2(G;\mathbb{Z}/p) \ar@{->}[r]
 & E \ar@{->}[r] & G \ar@{->}[r] & 1. }
\end{equation}
\end{cor}

\begin{proof}
Consider the universal $p$-covering of $G,$
$$\xymatrix@C=10mm@R=10mm{
 0 \ar@{->}[r] & H_2(G;\mathbb{Z}/p) \ar@{->}[r] & E \ar@{->}[r] & G \ar@{->}[r] & 1. }$$
with associated cohomology class $v_p.$
By Lemma \eqref{lema_eq_cocy}, we have that $(H_2(\phi; \mathbb{Z}/p))_*(v_p)=\phi^*(v_p),$ for every $\phi\in Aut(G).$
Equivalently, in terms of extensions, we have that there exists an element $\Phi\in Hom(E,E)$ making the following diagram commutative
\begin{equation*}
\xymatrix@C=10mm@R=10mm{
 0 \ar@{->}[r] & H_2(G;\mathbb{Z}/p) \ar@{->}[r] \ar@{->}[d]^-{H_2(\phi; \mathbb{Z}/p)} & E \ar@{->}[r] \ar@{->}[d]^-{\Phi} & G \ar@{->}[r] \ar@{->}[d]^-{\phi} & 1\\
0  \ar@{->}[r] & H_2(G;\mathbb{Z}/p) \ar@{->}[r]
 & E \ar@{->}[r] & G \ar@{->}[r] & 1. }
\end{equation*}
Finally, the $5$-lemma implies that $\Phi$ is an automorphism of $E$ that lifts $\phi$ as desired.
\end{proof}

\subsection{Construction of the extensions}
Next we apply this theory to the $p$-coverings of our interest, which are those given in diagram \eqref{p-central_ext_push-out},
in order to get a functorial construction of an extension of groups
\begin{equation}
\label{ses_aut_p}
\xymatrix@C=7mm@R=10mm{0 \ar@{->}[r] & Hom(\mathcal{N}^\bullet_1, \mathcal{L}^\bullet_{k+1}) \ar@{->}[r]^-{i} & Aut\;\mathcal{N}^\bullet_{k+1} \ar@{->}[r]^-{\psi_k^\bullet} & Aut \;\mathcal{N}^\bullet_k \ar@{->}[r] & 1 .}
\end{equation}
Let $\Gamma$ be a free group of finite rank, denote by $\{ \Gamma_k^\bullet\}_k$ the Stallings or Zassenhaus $p$-central series.
Given an arbitrary group $G,$ using the properties of Stallings and Zassenhaus $p$-central series one gets that the groups $G^\bullet_k$ are characteristic subgroups of $G.$ In particular, by Proposition \ref{prop_ident_L_versal}, the $p$-elementary abelian groups
$\mathcal{L}_{k}^\bullet,$ $\widetilde{\mathcal{L}_{k}^Z},$ $\frac{\Gamma^Z_{k+1}}{[\Gamma, \Gamma^Z_{k}](\Gamma^Z_{k})^p}$ are respectively characteristic subgroups of $\mathcal{N}_{k}^\bullet,$ $\widetilde{\mathcal{N}_{k}^Z},$ $\widetilde{\mathcal{N}_{k}^Z}.$

Then every automorphism $\widetilde{\Phi}$ of $\widetilde{\mathcal{N}_{k+1}^\bullet}$ induces a commutative diagram
\[
\begin{tikzcd}[column sep={1.2cm,between origins},row sep=0.5cm]
0 \arrow[rr] &&\widetilde{\mathcal{L}_{k+1}^\bullet} \arrow[rr]\arrow{rd}{\widetilde{\psi}} \arrow[dd] && \widetilde{\mathcal{N}_{k+1}^\bullet} \arrow[rr] \arrow{rd}{\widetilde{\Phi}} \arrow[dd] && \mathcal{N}^\bullet_k \arrow[rr] \arrow{rd}{\phi} \arrow[dd, equal] && 1 &  \\
& 0 \arrow[rr, crossing over] && \widetilde{\mathcal{L}_{k+1}^\bullet} \arrow[rr, crossing over]  && \widetilde{\mathcal{N}_{k+1}^\bullet} \arrow[rr, crossing over] && \mathcal{N}^\bullet_k \arrow[rr] && 1  \\
0 \arrow[rr] && \mathcal{L}_{k+1}^\bullet \arrow[rr] \arrow{rd}{\psi} &&\mathcal{N}_{k+1}^\bullet \arrow[rr] \arrow{rd}{\Phi} && \mathcal{N}^\bullet_k \arrow[rr] \arrow{rd}{\phi} && 1 & \\
& 0 \arrow[rr] &&\mathcal{L}_{k+1}^\bullet \arrow[rr] \arrow[uu, crossing over, leftarrow] && \mathcal{N}_{k+1}^\bullet \arrow[rr] \arrow[uu, crossing over, leftarrow] && \mathcal{N}^\bullet_k \arrow[rr] \arrow[uu, crossing over, equal] && 1,
\end{tikzcd}
\]
which induces a push-out diagram
\begin{equation}
\label{ses_aut_p_tilde}
\xymatrix@C=7mm@R=10mm{0 \ar@{->}[r] & Hom(\mathcal{N}^\bullet_1, \widetilde{\mathcal{L}^\bullet_{k+1}}) \ar@{->}[d] \ar@{->}[r]^-{i} & Aut\;\widetilde{\mathcal{N}^\bullet_{k+1}} \ar@{->}[d] \ar@{->}[r]^-{\widetilde{\psi^\bullet_k}} & Aut \;\mathcal{N}^\bullet_k \ar@{=}[d] \ar@{->}[r] & 1 \\
0 \ar@{->}[r] & Hom(\mathcal{N}^\bullet_1,\mathcal{L}^\bullet_{k+1}) \ar@{->}[r]^-{i} & Aut\;\mathcal{N}^\bullet_{k+1} \ar@{->}[r]^-{\psi^\bullet_k} & Aut \;\mathcal{N}^\bullet_k \ar@{->}[r] & 1. }
\end{equation}

\begin{rem}
Notice that the top and bottom short exact sequence of diagram
\eqref{ses_aut_p_tilde} coincide for $\bullet=S$ and differ for $\bullet=Z.$
\end{rem}

Next we construct the top extension of diagram \eqref{ses_aut_p_tilde}. We prove that $\widetilde{\psi^\bullet_k}$ is an epimorphism and
that $Ker(\widetilde{\psi^\bullet_k})$ is isomorphic to $Hom(\mathcal{N}^\bullet_1,\widetilde{\mathcal{L}_{k+1}^\bullet}).$

\vspace{0.5cm}
\textbf{The map $\widetilde{\psi^\bullet_k}$ is an epimorphism.}
Consider the central extension

\begin{equation*}
\xymatrix@C=7mm@R=10mm{1 \ar@{->}[r] & \widetilde{\mathcal{L}_{k+1}^\bullet} \ar@{->}[r] & \widetilde{\mathcal{N}^\bullet_{k+1}} \ar@{->}[r] & \mathcal{N}^\bullet_{k} \ar@{->}[r] & 1. }
\end{equation*}
By mod p Hopf isomorphism this extension becomes
\begin{equation*}
\xymatrix@C=7mm@R=10mm{1 \ar@{->}[r] & H_2(\mathcal{N}^\bullet_{k};\mathbb{Z}/p) \ar@{->}[r] & \widetilde{\mathcal{N}^\bullet_{k+1}} \ar@{->}[r] & \mathcal{N}^\bullet_{k} \ar@{->}[r] & 1. }
\end{equation*}
By Propositon \ref{prop_univ_p-cov_stall_zass} the above extension is the universal $p$-covering of $\mathcal{N}^\bullet_k.$
As a consequence, by Corollary \ref{cor_ver_ext_mod_p}, we get that 
$\widetilde{\psi^\bullet_k}: Aut (\widetilde{\mathcal{N}_{k+1}^\bullet}) \rightarrow Aut(\mathcal{N}^\bullet_k)$ is an epimorphism.

\vspace{0.5cm}
\textbf{The Kernel of $\widetilde{\psi^\bullet_k}$ is isomorphic to
$Hom(\mathcal{N}_{1}^\bullet,\widetilde{\mathcal{L}_{k+1}^\bullet}).$}

\begin{defi}[Section 9.1.3. in \cite{rot}]
An automorphism $\varphi$ of a group $E$ \textit{stabilizes} an extension
\begin{equation}
\xymatrix@C=7mm@R=10mm{0 \ar@{->}[r] & A  \ar@{->}[r] &  E \ar@{->}[r] & G \ar@{->}[r]& 1  }
\end{equation}
if the following diagram commutes:
\begin{equation}
\xymatrix@C=7mm@R=10mm{0 \ar@{->}[r] & A \ar@{->}[r]^-{i} \ar@{=}[d]&  E \ar@{->}[r]^-{p} \ar@{->}[d]^-{\varphi} & G \ar@{=}[d] \ar@{->}[r]& 1 \\
0 \ar@{->}[r] & A \ar@{->}[r]^-{i} & E \ar@{->}[r]^-{p} & G \ar@{->}[r] & 1. }
\end{equation}
The set of all stabilizing automorphisms of an extension of $A$ by $G,$ where $A$ is a $G$-module, is a group under composition and it is denoted by $Stab(G,A).$
In addition, by Corollary 9.16 in \cite{rot}, $Stab(G,A)$ is isomorphic to the group of derivations $Der(G,A)$ via the homomorphism
\begin{align*}
\sigma: Stab(G,A) & \rightarrow Der(G,A) \\
\varphi & \mapsto (d:\;G \rightarrow A),
\end{align*}
where $d(x)=\varphi(s(x))-s(x)$ with $s$ a section.
\end{defi}

Consider the universal $p$-covering of $\mathcal{N}^\bullet_k,$
\begin{equation}
\label{ext_tilde_stab}
\xymatrix@C=7mm@R=10mm{
0 \ar@{->}[r] &H_2(\mathcal{N}_{k}^\bullet;\mathbb{Z}/p) \ar@{->}[r] &\widetilde{\mathcal{N}_{k+1}^\bullet} \ar@{->}[r] & \mathcal{N}^\bullet_k \ar@{->}[r] & 1 .}
\end{equation}
Denote $v_p\in H^2(\mathcal{N}_{k}^\bullet;H_2(\mathcal{N}_{k}^\bullet;\mathbb{Z}/p))$ its associated cohomology class.
Let $\Phi\in Ker(\widetilde{\psi_k^\bullet}).$ Then $\Phi$ induces a commutative diagram
\begin{equation}
\label{com_diag_push-out_tilde}
\xymatrix@C=7mm@R=10mm{0 \ar@{->}[r] &H_2(\mathcal{N}_{k}^\bullet;\mathbb{Z}/p) \ar@{->}[r] \ar@{->}[d]^{\psi}&  \widetilde{\mathcal{N}_{k+1}^\bullet} \ar@{->}[r] \ar@{->}[d]^{\Phi}& \mathcal{N}^\bullet_k \ar@{=}[d] \ar@{->}[r]& 1 \\
0 \ar@{->}[r] & H_2(\mathcal{N}_{k}^\bullet;\mathbb{Z}/p) \ar@{->}[r] &\widetilde{\mathcal{N}_{k+1}^\bullet} \ar@{->}[r] & \mathcal{N}^\bullet_k \ar@{->}[r] & 1. }
\end{equation}
This implies that
$\psi_*(v_p)=v_p.$
Applying the natural isomorphism $\eta$ of \eqref{iso_uct_h2} to the above equality, we obtain that
$id=\eta(v_p)=\eta \psi_*(v_p)=\psi_*\eta(v_p)=\psi_*(id)=\psi.$
Therefore $\Phi$ stabilizes the extension \eqref{ext_tilde_stab}.
Hence,
\begin{align*}
Ker(\widetilde{\psi^\bullet_k})\cong & Stab(\mathcal{N}^\bullet_{k},H_2(\mathcal{N}^\bullet_k;\mathbb{Z}/p))\cong  Der(\mathcal{N}^\bullet_{k},H_2(\mathcal{N}^\bullet_k;\mathbb{Z}/p))= \\ =& Hom(\mathcal{N}^\bullet_{k},H_2(\mathcal{N}^\bullet_k;\mathbb{Z}/p))= Hom(\mathcal{N}^\bullet_1,H_2(\mathcal{N}^\bullet_k;\mathbb{Z}/p))\cong \\
\cong & Hom(\mathcal{N}^\bullet_1,\widetilde{\mathcal{L}^\bullet_{k+1}}),
\end{align*}
where the last isomorphism follows from the mod p Hopf formula.
Therefore we get the following result:

\begin{prop}
\label{prop_exact_seq_modp}
We have an exact sequence of groups
\begin{equation*}
\xymatrix@C=7mm@R=10mm{0 \ar@{->}[r] & Hom(\mathcal{N}^\bullet_1, \widetilde{\mathcal{L}^\bullet_{k+1}}) \ar@{->}[r]^-{i} & Aut\;\widetilde{\mathcal{N}^\bullet_{k+1}} \ar@{->}[r]^-{\widetilde{\psi^\bullet_k}} & Aut \;\mathcal{N}^\bullet_k \ar@{->}[r] & 1, }
\end{equation*}
where $i$ is defined as $i(f)= (\gamma\mapsto f([\gamma])\gamma).$
\end{prop}

Now taking a push out diagram
\begin{equation*}
\xymatrix@C=7mm@R=10mm{0 \ar@{->}[r] & Hom(\mathcal{N}^\bullet_1, \widetilde{\mathcal{L}^\bullet_{k+1}}) \ar@{->}[d]^-{q} \ar@{->}[r]^-{i} & Aut\;\widetilde{\mathcal{N}^\bullet_{k+1}} \ar@{->}[d] \ar@{->}[r]^-{\widetilde{\psi^\bullet_k}} & Aut \;\mathcal{N}^\bullet_k \ar@{=}[d] \ar@{->}[r] & 1 \\
0 \ar@{->}[r] & Hom(\mathcal{N}^\bullet_1,\mathcal{L}^\bullet_{k+1}) \ar@{->}[r]^-{i} & Aut\;\mathcal{N}^\bullet_{k+1} \ar@{->}[r]^-{\psi^\bullet_k} & Aut \;\mathcal{N}^\bullet_k \ar@{->}[r] & 1, }
\end{equation*}
where $q$ is the quotient respect $\frac{\Gamma^\bullet_{k+2}}{[\Gamma,\Gamma^\bullet_{k+1}](\Gamma^\bullet_{k+1})^p},$ we have the expected result:
\begin{cor}
\label{cor_exact_seq_modp_L}
We have an exact sequence of groups
\begin{equation*}
\xymatrix@C=7mm@R=10mm{0 \ar@{->}[r] & Hom(\mathcal{N}^\bullet_1,\mathcal{L}^\bullet_{k+1}) \ar@{->}[r]^-{i} & Aut\;\mathcal{N}^\bullet_{k+1} \ar@{->}[r]^-{\psi^\bullet_k} & Aut \;\mathcal{N}^\bullet_k \ar@{->}[r] & 1, }
\end{equation*}
where $i$ is defined as $i(f)= (\gamma\mapsto f([\gamma])\gamma).$
\end{cor}

\noindent \textbf{Comparing $Aut(\mathcal{N}^Z_k)$ and $Aut(\mathcal{N}^S_k)$.}
Notice that, in general, the extensions of Corollary \eqref{cor_exact_seq_modp_L} with $\bullet=Z$ and $\bullet=S$ are distinct. However, by Proposition 2.6 in \cite{coop}, for every positive integer $l,$ we have that
$\Gamma^Z_{p^l}<\Gamma^S_l<\Gamma^Z_l.$ Moreover one can check that the groups $\frac{\Gamma^Z_l}{\Gamma^Z_{p^l}},$ $\frac{\Gamma^S_l}{\Gamma^Z_{p^l}},$ $\frac{\Gamma^Z_l}{\Gamma^S_l}$ are respectively characteristic subgroups of $\mathcal{N}^Z_{p^l-1},$ $\mathcal{N}^Z_{p^l-1},$ $\mathcal{N}^S_{l-1}.$
Therefore we have a commutative diagram
\begin{equation*}
\xymatrix@C=7mm@R=10mm{Aut (\mathcal{N}_{p^l-1}^Z) \ar@{->}[r] \ar@{->}[d]& Aut(\mathcal{N}^Z_{l-1}) \ar@{=}[d]  \\
Aut(\mathcal{N}_{l-1}^S) \ar@{->}[r] & Aut(\mathcal{N}^Z_{l-1}) ,}
\end{equation*}
which allows us to compare the aforementioned extensions.

\subsection{Centrality}
Next we study the centrality of the extensions constructed in the previous Section.
Given an arbitrary group $G,$  we denote by $IA_k^\bullet(G)$ the elements of $Aut(G)$ that act trivially on $G/G^\bullet_{k+1},$ that is,
$IA_k^\bullet(G)=\{f\in Aut(G)\;\mid \; f(x)x^{-1}\in G^\bullet_{k+1} \text{ for all } x\in G\}.$
Throughout this Section we denote $IA_1^\bullet(G)$ by $IA^p(G),$ with $\bullet= S \text{ or }Z.$

In the sequel we show that the extension of groups
\begin{equation}
\label{ses_aut_ZS}
\xymatrix@C=7mm@R=10mm{0 \ar@{->}[r] & Hom(\mathcal{N}^\bullet_1,\mathcal{L}^\bullet_{k+1}) \ar@{->}[r] & Aut\;\mathcal{N}^\bullet_{k+1} \ar@{->}[r] & Aut \;\mathcal{N}^\bullet_k \ar@{->}[r] & 1}
\end{equation}
is a non-central extension, but that if we restrict this extension to $IA^p(\mathcal{N}^\bullet_k)$ we get another extension of groups
\begin{equation}
\label{ses_ZS}
\xymatrix@C=7mm@R=10mm{0 \ar@{->}[r] & Hom(\mathcal{N}_1^\bullet, \mathcal{L}^\bullet_{k+1}) \ar@{->}[r] & IA^p(\mathcal{N}^\bullet_{k+1}) \ar@{->}[r] & IA^p(\mathcal{N}^\bullet_k) \ar@{->}[r] & 1,}
\end{equation}
which is a central extension.
The main argument to get these results is based on the fact that, as we will see, the action of $Aut(\mathcal{N}^\bullet_{k+1})$ on $Hom(\mathcal{N}_1^\bullet, \mathcal{L}^\bullet_{k+1})$ factors through $Aut(\mathcal{N}^\bullet_{1}).$
We first give some preliminary results.

\begin{lema}[Three Subgroups Lemma,\cite{hall}]
Let $A,$ $B$ and $C$ be subgroups of a group $G.$ If $N\lhd G$ is a normal subgroup such that $[A,[B,C]]$ and $[B,[C,A]]$ are contained in $N$ then $[C,[A,B]]$ is also contained in $N.$
\end{lema}

To deal simultaneously with commutator calculus in both $G$ and $Aut\;G$ we introduce the Holomorph group of $G.$

\begin{defi}
Let $G$ be a group. The \emph{Holomorph group} of $G$ is defined as the semidirect product
$$Hol(G)=G\rtimes Aut(G),$$
where the multiplication is given by $(g_1,f_1)(g_2,f_2)=(g_1f_1(g_2),f_1f_2).$
\end{defi}

Throughout this Section, given $x\in G,$ $f\in Aut(G),$ $H\triangleleft G$ and $K\triangleleft Aut(G),$ we denote by $[f,x]\in G$ the element $[f,x]=f(x)x^{-1},$ and by $[K,H]$ the subgroup given by
$[K,H]=\{[f,x]\in G \; ; \; f\in K, x\in H \}.$

\begin{lema}
\label{lema_coop_general}
Given a prime $p.$ If $f\in IA_k^\bullet(G)$ and $x\in G^\bullet_l,$ then $f(x)x^{-1}\in G^\bullet_{k+l}.$ Equivalently, $[IA_k^\bullet(G),G^\bullet_l]<G^\bullet_{k+l}.$
\end{lema}

\begin{rem}
This Lemma is a generalization of Lemma 3.7 in \cite{coop}.
However, in the proof of that Lemma the author asserted that the result for the Zassenhaus filtration follows using the same argument that he used for the Stallings filtration.
Reviewing his proof we found that this is not clear.
To handle the Zassenhaus filtration case we rely instead on Hall identities with weights (see Section \ref{sec_comm}).
\end{rem}

\begin{proof}
\textbf{For the case of the Stallings series,} recall that $G^S_{l+1}=[G,G^S_l](G_l^S)^p.$ Thus every element of $G^S_{l+1}$ is a product of elements of the form $[x,y]\in G^S_{l+1},$ $z^p\in G^S_{l+1},$ where $x\in G$ and $y,z\in G^S_l.$
So first of all we will prove the statement for such elements and later for any product of them.

We proceed by induction on $l.$ The base case $l=1$ follows from the definition of $IA_k^S(G).$ Assume that the lemma holds for $l.$ We prove that this lemma also holds for $l+1.$

Consider elements of the form $[x,y]\in G^S_{l+1}$ and $z^p\in G^S_{l+1},$ where $x\in G$ and $y,z\in G^S_l.$

We first show that $f([x,y])[x,y]^{-1}\in G^S_{k+l+1}$ for $f\in IA_k^S(G).$ The main idea of this proof originally comes from \cite{andrea}. First note that
$$f([x,y])[x,y]^{-1}=[f,[x,y]]\in [IA_k^S(G),[G,G^S_l]].$$
The idea is to apply the Three Subgroup Lemma for the subgroups $IA_k^S(G),$ $G,$ $G^S_l$ of $Hol(G).$
Observe that by induction and the definition of the Stallings series,
\begin{align*}
[[IA_k^S(G),G^S_l],G] & <[G^S_{k+l},G]< G^S_{k+l+1}, \\
[[IA_k^S(G),G],G^S_l] & <[G^S_{k+1},G^S_l]< G^S_{k+l+1}.
\end{align*}
Moreover, since $G^S_{k+l+1}$ is a normal subgroup of $G,$ we can view $G^S_{k+l+1}$ as a normal subgroup of $Hol(G).$
Therefore the Three Subgroup Lemma implies that
$$[IA_k^S(G),[G,G^S_l]]<G^S_{k+l+1}.$$
Next we show that $f(z^p)z^{-p}\in G^S_{k+l+1}.$ We first prove that
$$f(z^p)z^{-p}\equiv (f(z)z^{-1})^p \text{ mod } G^S_{k+l+1}.$$
Observe that the following formula holds
\begin{equation}
\label{eq_z^p_stall}
f(z^p)z^{-p}=[f,z^p]=[f,z][f,z]^z\ldots [f,z]^{z^{p-1}}.
\end{equation}
By induction, $[f,z]\in G^S_{k+l}$ and by normality, $[f,z]^{z^i}\in G^S_{k+l}$ for each $i=1,\ldots, p-1.$ Furthermore,
$[[f,z],z^i]\in G^S_{k+l+1}$ so that $[f,z] \equiv [f,z]^{z^i} \text{ mod } G^S_{k+l+1}.$ As a consequence, by induction and the properties of the Stallings series, from formula \eqref{eq_z^p_stall} we get that $f(z^p)z^{-p}\equiv (f(z)z^{-1})^p \equiv 1 \text{ mod } G^S_{k+l+1}.$
Therefore, $f(z^p)z^{-p}\in G^S_{k+l+1}.$

Finally, we prove the statement for products of elements of $G_{l+1}^S.$
Let $f\in IA_k^S(G)$ and $\eta_i\in G^S_{l+1}.$ Using the fact that $f(\eta_i)\eta_i^{-1}\in G^S_{k+l+1}$ for all $i,$ we have:
\begin{align*}
f\left(\prod_{i=1}^n \eta_i\right)\left(\prod_{i=1}^n \eta_i\right)^{-1}= & f(\eta_1)\cdots f(\eta_{n-1})f(\eta_n)\eta_n^{-1}\eta_{n-1}^{-1}\cdots \eta_1 \\
\equiv & f(\eta_1)\cdots f(\eta_{n-1})\eta_{n-1}^{-1}\cdots \eta_1^{-1} \quad(\text{mod }G^S_{k+l+1})\\
\equiv & f(\eta_1)\cdots f(\eta_{n-2})\eta_{n-2}^{-1}\cdots \eta_1^{-1} \quad(\text{mod }G^S_{k+l+1})\\
\equiv & \cdots \equiv f(\eta_{1})f(\eta_{2})\eta_{2}^{-1}\eta_{1}^{-1}\equiv f(\eta_{1})\eta_1 \equiv 1\quad(\text{mod }G^S_{k+l+1}).
\end{align*}
Therefore,
$$f\left(\prod_{i=1}^n \eta_i\right)\left(\prod_{i=1}^n \eta_i\right)^{-1}\in G^S_{k+l+1}.$$

\textbf{For the case of the Zassenhaus series,} recall that
$G^Z_l=\prod_{ip^j\geq l}(G_i)^{p^j}.$ Thus every element of $G^Z_l$ is a product of elements of the form $x_i^{p^j}\in G^Z_l$ with $x_i\in G_i$ and $ip^j\geq l.$
So first of all we will prove the statement for such elements and later for any product of them.

We proceed by induction on $l.$ The base case $l=1$ follows from the definition of $IA_k^Z(G).$ Assume that the lemma holds for $l.$ We prove that this lemma also holds for $l+1.$

Consider $f\in I_k^Z(G)$ and $x_i^{p^j}\in G^Z_{l+1}$ with $x_i\in G_i,$ i.e. $ip^j\geq l+1.$ We want to show that $f(x_i^{p^j})x_i^{-p^j}\in G^Z_{l+k+1}.$

If $i\geq l+1,$ then $x_i\in G_{l+1}.$ Then we can rewrite $x_i$ as $[x,x_l]$ with $x\in G$ and $x_l \in G_l.$

As in the case of Stallings, applying the Three Subgroup Lemma for the subgroups $IA^Z_k(G),G,G_l< Hol(G),$ we have that
$[IA^Z_k(G),[G,G_l]]<G^Z_{k+l+1}.$
Then we have that
$$f(x_i^{p^j})=f(x_i)^{p^j}=f([x,x_l])^{p^j}\equiv [x,x_l]^{p^j}=x_i^{p^j} \text{ mod } G^Z_{k+l+1}.$$

If $i\leq l,$ then the condition $ip^j\geq l+1$ implies that $j\geq 1.$

Observe that
$f(x_i^{p^j})=f(x_i)^{p^{j}}.$
By induction hypothesis, $f(x_i)x^{-1}_i\in G^Z_{i+k}.$
Then there exists an element $y_{i+k}\in G^Z_{i+k}$ such that $f(x_i)=x_iy_{i+k}.$
Next we show that
$$(x_iy_{i+k})^p\equiv x_i^p \quad (\text{mod } G^Z_{l+k+1}).$$
By Theorem \eqref{teo_hall} we have that
$$(x_iy_{i+k})^{p}=R_1^{n_1}R_2^{n_2}\cdots R_r^{n_r}\cdots$$
where
$$R_1,R_2,\ldots ,R_r,\ldots \qquad (R_1=x_i,\quad R_2=y_{i+k})$$
are the various formally distinct complex commutators of $x_i$ and $y_{i+k}$ arranged in increasing weights order, and $n_1,n_2,\ldots ,n_r,\ldots$ positive integers such that $n_1=n_2=p^j$ and if the weight $w_r$ of $R_r$ in $x_i$ and $y_{i+k}$ satisfies $p^{\beta-1}\leq w_r<p^\beta< p^j$ then $n_r$ is divisible by $p^{\beta-1}.$

Next we prove that $R_r^{n_r}\in G^Z_{l+k+1}$ for $r\geq 2.$
\begin{itemize}
\item \textbf{For $r=2,$} we know that $R_2=y_{i+k}$ and $n_2=p^j.$
Since $y_{i+k}\in G^Z_{i+k},$ by the properties of Zassenhaus series we have that
$$y_{i+k}^{p^j}\in (G^Z_{i+k})^{p^j}\leq G^Z_{p^j(i+k)}\leq G^Z_{ip^j+k}\leq G^Z_{l+k+1}.$$
\item \textbf{For $r\geq 3,$} as $R_r$ are complex commutators of weight $w_r$ in the two components $x_k,$ $y_{i+k},$ we have that $w_r\geq 2.$
As a consequence,
at least one component of $R_r$ has to be $y_{i+k},$ because if it is not the case then $R_r$ has to be $1.$

If the weight $w_r$ of $R_r$ in $x_k$ and $y_{i+k}$ satisfies that $p^{\beta-1}\leq w_i<p^\beta ,$ since at least one component of $R_r$ has to be $y_{i+k},$ by the properties of the Zassenhaus series, we have that $R_r\in G^Z_{(\omega_r-1)i+(i+k)}=G^Z_{\omega_ri+k}\leq G^Z_{p^{\beta-1}i+k}.$
Moreover, in the case $p^{\beta-1}\leq w_r<p^\beta ,$ we have that $p^{j-\beta+1}\mid n_r$ and by the properties of the Zassenhaus series, we have that
$$R_r^{n_r}\in (G^Z_{\omega_ri+k})^{n_r}\leq (G^Z_{p^{\beta-1}i+k})^{p^{j-\beta+1}}\leq G^Z_{p^{j-\beta+1}(p^{\beta-1}i+k)}\leq G^Z_{ip^j+k}\leq G^Z_{l+k+1}.$$

On the other hand, if the weight $w_r$ of $R_r$ in $x_i$ and $y_{i+k}$ satisfies that $w_r\geq p^j,$ since at least one component of $R_r$ has to be $y_{i+k},$ by the properties of the Zassenhaus series, we have that
$R_r\in G^Z_{(\omega_r-1)i+(i+k)}=G^Z_{\omega_ri+k}\leq G^Z_{ip^j+k}\leq G^Z_{l+k+1}.$ As a consequence, $$R_r^{n_r}\in G^Z_{l+k+1}.$$
\end{itemize}

Therefore,
$(x_iy_{i+k})^{p^j}\equiv x_i^{p^j} \quad (\text{ mod } G^Z_{l+k+1} ),$
i.e. $f(x_i^{p^j})x_i^{-p^j}\in G^Z_{l+k+1}.$

Finally, we prove the statement for products. The same argument of the case of Stallings works here, obtaining that if 
$f\in IA^Z_k(G)$ and $\eta_i\in G^Z_{l+1},$ then
$$f\left(\prod_{i=1}^n \eta_i\right)\left(\prod_{i=1}^n \eta_i\right)^{-1}\in G^S_{l+k+1}.$$
\end{proof}

As a direct consequence of Lemma \eqref{lema_coop_general}, using ideas of S. Andreadakis (see Theorem 1.1 in \cite{andrea}),
we get the following result.
\begin{cor}
\label{cor_IA_bullet}
For any two elements $\varphi\in IA^\bullet_k(G)$ and $\psi\in IA^\bullet_l(G),$ the commutator $[\varphi, \psi]$ is contained in $IA^\bullet_{k+l}(G).$ Equivalently, $[IA^\bullet_k(G),IA^\bullet_l(G)] < IA^\bullet_{k+l}(G).$
\end{cor}

\begin{proof}
Notice that $G,$ $IA^\bullet_k(G),$ $IA^\bullet_l(G)$ are subgroups of $Hol(G).$ By Lemma \eqref{lema_coop_general},
\begin{align*}
[IA^\bullet_l(G),[IA^\bullet_k(G),G]] & < [IA^\bullet_l(G),G^\bullet_{k+1}]< G^\bullet_{k+l+1}, \\
[IA^\bullet_k(G),[G, IA^\bullet_l(G)]] & < [IA^\bullet_k(G),G^\bullet_{l+1}]< G^\bullet_{k+l+1}.
\end{align*}
Using the Three Subgroups Lemma, we get that $[[IA^\bullet_k(G),IA^\bullet_l(G)],G]< G^\bullet_{k+l+1}.$

\end{proof}

Going back to our main problem we now show the expected result:

\begin{prop}
\label{prop_act_aut_bullet}
The natural action of $Aut(\mathcal{N}^\bullet_{k+1})$ on $Hom(\mathcal{N}_1^\bullet, \mathcal{L}^\bullet_{k+1})$
factors through $Aut(\mathcal{N}^\bullet_{1}).$
\end{prop}

\begin{proof}
In virtue of Corollary \eqref{cor_exact_seq_modp_L}, the natural action of $Aut(\mathcal{N}^\bullet_{k+1})$ on $Hom(\mathcal{N}_1^\bullet, \mathcal{L}^\bullet_{k+1})$ is given by
\begin{align*}
Aut(\mathcal{N}^\bullet_{k+1})\times Hom(\mathcal{N}_1^\bullet, \mathcal{L}^\bullet_{k+1})&\longrightarrow Hom(\mathcal{N}_1^\bullet, \mathcal{L}^\bullet_{k+1}) \\
(h,f) & \longmapsto (x\mapsto h(f(h^{-1}x)))),
\end{align*}
where $h^{-1}x$ is the action of $h^{-1}\in Aut(\mathcal{N}^\bullet_{k+1})$ on $x\in \mathcal{N}^\bullet_1$ via the surjection
$Aut\;\mathcal{N}^\bullet_{k} \rightarrow Aut \;\mathcal{N}^\bullet_1.$
Moreover, by Proposition \eqref{prop_ident_L_versal} and Lemma \eqref{lema_coop_general}, we know that if $h\in IA^p(\mathcal{N}^\bullet_{k+1})$ and $y\in \mathcal{L}_{k+1}^\bullet,$ then
$h(y)=y.$ Therefore the action of $Aut\;\mathcal{N}^\bullet_{k}$ on $\mathcal{L}_{k+1}^\bullet$ factors through $Aut(\mathcal{N}^\bullet_{1})$
via the surjection $Aut\;\mathcal{N}^\bullet_{k} \rightarrow Aut \;\mathcal{N}^\bullet_1$ and we get the result.
\end{proof}

As a consequence we have that

\begin{prop}
Let $\Gamma$ be a free group of finite rank $n > 1.$ The extension
\begin{equation*}
\xymatrix@C=7mm@R=10mm{0 \ar@{->}[r] & Hom(\mathcal{N}^\bullet_1,\mathcal{L}^\bullet_{k+1}) \ar@{->}[r]^-{i} & Aut\;\mathcal{N}^\bullet_{k+1} \ar@{->}[r]^-{\psi^\bullet_k} & Aut \;\mathcal{N}^\bullet_k \ar@{->}[r] & 1, }
\end{equation*}
is not central.
\end{prop}

\begin{proof}
Let us provide a counterexample to the centrality.

Consider $\Gamma=\langle x_1,\ldots, x_n \rangle$ and $x=[x_1,[x_2,[x_1,[x_2,\ldots]]]]$ a commutator of length $k+1.$
Denote $[x_i]$ the class of $x_i$ in $\mathcal{N}^\bullet_1,$ and $\overline{x}$ the class of $x$ in $\mathcal{L}^\bullet_{k+1}.$
Observe that $\mathcal{N}_1^\bullet$ is a $\mathbb{Z}/p$-vector space with basis $\{[x_1],\ldots ,[x_n]\}.$
Let $f\in Hom(\mathcal{N}_1^\bullet,\mathcal{L}^\bullet_{k+1})$ be the homomorphism defined on the basis of $\mathcal{N}_1^\bullet$ by
$f([x_1])=\overline{x}$ and $f([x_i])=0$ for $1<i\leq n.$

Consider $h=(x_1x_2)\in \mathfrak{S}_n\subset Aut \; \Gamma.$
Since $\Gamma^\bullet_{k+2}$ is a characteristic subgroup of $\Gamma,$
$h$ induces an element $\overline{h}\in Aut\;\mathcal{N}^\bullet_{k+1}.$ 
Then we have that
$$f([x_1])=\overline{x}\quad \text{and} \quad \overline{h}(f((\overline{h})^{-1}[x_1]))=\overline{h}(f([x_2]))=\overline{h}(1)=1.$$
Therefore the extension is not central.
\end{proof}

Whereas, if we restrict the extension of this Proposition to $IA^p(\mathcal{N}^\bullet_{k+1}),$ we get that

\begin{prop}
\label{prop_cent_ext_IA}
Let $\Gamma$ be a free group of finite rank $n > 1.$ The extension
\begin{equation*}
\xymatrix@C=7mm@R=10mm{0 \ar@{->}[r] & Hom(\mathcal{N}^\bullet_1, \mathcal{L}^\bullet_{k+1}) \ar@{->}[r]^-{i} & IA^p(\mathcal{N}^\bullet_{k+1}) \ar@{->}[r]^-{\pi} & IA^p(\mathcal{N}^\bullet_k) \ar@{->}[r] & 1 }
\end{equation*}
is central.
\end{prop}

\section{Splitting the extensions}
\label{sec_split}
We now discuss whenever the extensions split:
\begin{align*}
& \xymatrix@C=7mm@R=10mm{0 \ar@{->}[r] & Hom(\mathcal{N}^\bullet_1,\mathcal{L}^\bullet_{k+1}) \ar@{->}[r]^-{i} & Aut\;\mathcal{N}^\bullet_{k+1} \ar@{->}[r]^-{\psi^\bullet_k} & Aut \;\mathcal{N}^\bullet_k \ar@{->}[r] & 1, }
\\
& \xymatrix@C=7mm@R=10mm{0 \ar@{->}[r] & Hom(\mathcal{N}^\bullet_1, \mathcal{L}^\bullet_{k+1}) \ar@{->}[r]^{i} & IA^p(\mathcal{N}^\bullet_{k+1}) \ar@{->}[r]^{\psi^\bullet_k} & IA^p(\mathcal{N}^\bullet_k) \ar@{->}[r] & 1 .}
\end{align*}
Throughout this Section, given a prime number $p$ and $n\geq 2$ an integer, we denote by $GL_n(\mathbb{Z}/p)$ the general linear group of degree $n$ over $\mathbb{Z}/p$ and by $\mathfrak{gl}_{n}(\mathbb{Z}/p)$ its associated Lie algebra, i.e. the additive group of matrices $n\times n$ with coefficients in $\mathbb{Z}/p,$ by $SL_n(\mathbb{Z}/p)$ the special linear group of degree $n$ over $\mathbb{Z}/p$ and by $\mathfrak{sl}_{n}(\mathbb{Z}/p)$ its associated Lie algebra, i.e. the subgroup of $\mathfrak{gl}_{n}(\mathbb{Z}/p)$ formed by the matrices of trace zero, and by $H_p$ the group
$\mathcal{N}_1^S=\mathcal{N}_1^Z=\Gamma/\Gamma^p\Gamma_2.$

We first proceed to study the case $k=1.$ In order to deal with this case, we will use the Center Kills Lemma, a result about the transfer maps and $p$-Sylow groups (c.f. Section III.10 in \cite{brown}) and some computations on $H^2(SL_n(\mathbb{Z}/p);\mathfrak{sl}_{n}(\mathbb{Z}/p)).$
Then, to treat the case $k\geq 2,$ we develop some computations of commutators on $IA^p(\mathcal{N}_k^\bullet),$ we show that the second extension does not split and taking a push-out diagram we will get that for $k\geq 2$ the first extension does not split too.

\begin{teo}[Theorem 7 in \cite{chih}]
\label{teo_SL_split}
Let $p$ be a prime number and $n\geq 2$ an integer. The extension
$$\xymatrix@C=7mm@R=10mm{ 0 \ar@{->}[r] & \mathfrak{sl}_{n}(\mathbb{Z}/p) \ar@{->}[r] &  SL_{n}(\mathbb{Z}/p^2) \ar@{->}[r]^{r_p} &  SL_{n}(\mathbb{Z}/p) \ar@{->}[r] & 1,}$$
only splits for $(p,n)=(3,2)\text{ and }(2,3).$
\end{teo}
\begin{cor}
\label{cor_put_GL_split}
Let $p$ be a prime number and $n\geq 2$ an integer. The extension
$$\xymatrix@C=7mm@R=10mm{ 0 \ar@{->}[r] & \mathfrak{gl}_{n}(\mathbb{Z}/p) \ar@{->}[r] &  GL_{n}(\mathbb{Z}/p^2) \ar@{->}[r]^{r_p} &  GL_{n}(\mathbb{Z}/p) \ar@{->}[r] & 1,}$$
only splits for $(p,n)=(3,2),\; (2,2) \text{ and } (2,3).$
\end{cor}

\begin{proof}
We first prove that the extension
\begin{equation}
\label{ext_GL_split}
\xymatrix@C=7mm@R=10mm{ 0 \ar@{->}[r] & \mathfrak{gl}_{n}(\mathbb{Z}/p) \ar@{->}[r] &  GL_{n}(\mathbb{Z}/p^2) \ar@{->}[r]^{r_p} &  GL_{n}(\mathbb{Z}/p) \ar@{->}[r] & 1}
\end{equation}
does not split for $(p,n)\neq(3,2),\;(2,2) \text{ and } (2,3).$
Set
$$SL_{n}^{(p)}(\mathbb{Z}/p^2)=\{ A\in GL_{n}(\mathbb{Z}/p^2) \mid det(A)\equiv 1 \;(\text{mod }p) \}.$$
By construction we have a commutative diagram
\begin{equation}
\xymatrix@C=7mm@R=10mm{ 0 \ar@{->}[r] & \mathfrak{sl}_{n}(\mathbb{Z}/p) \ar@{^{(}->}[d]^-{i} \ar@{->}[r] &  SL_{n}(\mathbb{Z}/p^2)\ar@{^{(}->}[d] \ar@{->}[r] &  SL_{n}(\mathbb{Z}/p) \ar@{=}[d] \ar@{->}[r] & 1 \\
0 \ar@{->}[r] & \mathfrak{gl}_{n}(\mathbb{Z}/p) \ar@{=}[d] \ar@{->}[r] &  SL^{(p)}_{n}(\mathbb{Z}/p^2) \ar@{^{(}->}[d] \ar@{->}[r] &  SL_{n}(\mathbb{Z}/p) \ar@{^{(}->}[d]^-{\phi}  \ar@{->}[r] & 1 \\
0 \ar@{->}[r] & \mathfrak{gl}_{n}(\mathbb{Z}/p) \ar@{->}[r] &  GL_{n}(\mathbb{Z}/p^2) \ar@{->}[r] &  GL_{n}(\mathbb{Z}/p) \ar@{->}[r] & 1.}
\end{equation}
Let $g_n$ and $s_n$ be the associated cohomology class of the bottom and top extensions of this diagram, respectively. At the level of cohomology groups we know that
$$\phi^*(g_n)=i_*(s_n)\in H^2(SL_{n}(\mathbb{Z}/p);\mathfrak{gl}_{n}(\mathbb{Z}/p)).$$
By Theorem \ref{teo_SL_split}, $s_n\neq 0$ for $(p,n)\neq(3,2),(2,3).$ Then it is enough to show that the map
$$i_*: H^2(SL_{n}(\mathbb{Z}/p);\mathfrak{sl}_{n}(\mathbb{Z}/p))\longrightarrow H^2(SL_{n}(\mathbb{Z}/p);\mathfrak{gl}_{n}(\mathbb{Z}/p))$$
is injective for $(p,n)\neq(3,2),\;(2,2) \text{ and } (2,3).$

Consider the short exact sequence
$$\xymatrix@C=7mm@R=10mm{ 0 \ar@{->}[r] & \mathfrak{sl}_{n}(\mathbb{Z}/p) \ar@{->}[r] &  \mathfrak{gl}_{n}(\mathbb{Z}/p) \ar@{->}[r]^-{tr} & \mathbb{Z}/p \ar@{->}[r] & 1,}$$
where $tr$ is given by the matrix trace.
The long cohomology sequence for $SL_{n}(\mathbb{Z}/p)$ with values in above short exact sequence, give us an exact sequence
$$\xymatrix@C=7mm@R=10mm{ H^1(SL_{n}(\mathbb{Z}/p);\mathbb{Z}/p) \ar@{->}[r] &  H^2(SL_{n}(\mathbb{Z}/p);\mathfrak{sl}_{n}(\mathbb{Z}/p)) \ar@{->}[r]^-{i_*} & H^2(SL_{n}(\mathbb{Z}/p);\mathfrak{gl}_{n}(\mathbb{Z}/p)).}$$
Taking the presentation of $SL_{n}(\mathbb{Z}/p)$ given in \cite{green}, one can see that
$H^1(SL_{n}(\mathbb{Z}/p);\mathbb{Z}/p)=Hom(SL_{n}(\mathbb{Z}/p),\mathbb{Z}/p)=0$ except for $(p,n)=(2,2), (3,2),$ in which cases it is $\mathbb{Z}/p.$
As a consequence, $i_*$ is injective for $(p,n)\neq (2,2), (3,2).$

Next we prove that the extension \eqref{ext_GL_split} splits for $(p,n)=(3,2),\; (2,2) \text{ and } (2,3).$
By Proposition 4.5 in \cite{chih2}, we know that $H^2(GL_2(\mathbb{Z}/p);\mathfrak{gl}_2(\mathbb{Z}/p))=0$ for $p=2,3.$ Therefore the extension \eqref{ext_GL_split} splits for $(p,n)=(3,2) \text{ and } (2,2).$

For the case $(p,n)=(2,3),$ consider the push-out diagram
\begin{equation}
\label{diag_push-out_M_(2,3)}
\xymatrix@C=7mm@R=10mm{ 0 \ar@{->}[r] & \mathfrak{sl}_{3}(\mathbb{Z}/2) \ar@{^{(}->}[d]^{i} \ar@{->}[r] &  SL_{3}(\mathbb{Z}/4)\ar@{^{(}->}[d] \ar@{->}[r]^{r_2} &  SL_{3}(\mathbb{Z}/2) \ar@{=}[d] \ar@{->}[r] & 1 \\
0 \ar@{->}[r] & \mathfrak{gl}_{3}(\mathbb{Z}/2) \ar@{->}[r] &  SL^{(2)}_{3}(\mathbb{Z}/4) \ar@{->}[r]^{r_2} &  SL_{3}(\mathbb{Z}/2) \ar@{->}[r] & 1.}
\end{equation}
By Theorem \eqref{teo_SL_split} we know that the top extension of this commutative diagram splits. Then by commutative diagram \eqref{diag_push-out_M_(2,3)}, the bottom extension in diagram \eqref{diag_push-out_M_(2,3)} splits too.
Notice that $SL^{(2)}_{3}(\mathbb{Z}/4)=GL_3(\mathbb{Z}/4)$ and $SL_3(\mathbb{Z}/2)=GL_3(\mathbb{Z}/2).$ Hence, the extension \eqref{ext_GL_split} splits for $(p,n)=(2,3).$
\end{proof}

\begin{prop}
Let $p$ be a prime number. The extension
\begin{equation}
\label{ext_split_1}
\xymatrix@C=7mm@R=10mm{0 \ar@{->}[r] & Hom(\mathcal{N}^\bullet_1,\mathcal{L}^\bullet_{2}) \ar@{->}[r]^-{i} & Aut\;\mathcal{N}^\bullet_{2} \ar@{->}[r]^-{\psi^\bullet_1} & Aut \;\mathcal{N}^\bullet_1 \ar@{->}[r] & 1, }
\end{equation}
splits if and only if $p$ is odd if $\bullet=Z$  or $(p,n)= (3,2), (2,2)$ if $\bullet=S$.
\end{prop}

\begin{rem}
Notice that for $p=2,$ the extensions \eqref{ext_split_1} with $\bullet=Z,$ $\bullet=S$ coincide, because in this case,
$\Gamma_i^S=\Gamma_i^Z$ for $i=1,2,3.$
Since $\Gamma_3^Z=\Gamma_3\Gamma_2^2\Gamma^4 \subset [\Gamma,\Gamma^2]\Gamma^4=\Gamma^S_3,$ and we already know that $\Gamma^S_i\subset \Gamma^Z_i.$
\end{rem}

\begin{proof}
\textbf{Case $\bullet=Z$ and $p$ an odd prime.} In this case, we have that
$$\mathcal{L}^Z_{2}\cong\extp^2 \mathcal{N}^Z_1=\extp^2 H_p, \qquad Aut\; \mathcal{N}^Z_1 =Aut\; H_p\cong GL_{n}(\mathbb{Z}/p).$$
Then, the extension \eqref{ext_split_1} becomes
\begin{equation*}
\xymatrix@C=7mm@R=10mm{0 \ar@{->}[r] & Hom(H_p,\extp^2 H_p) \ar@{->}[r]^-{i} & Aut\;\mathcal{N}^Z_{2} \ar@{->}[r]^-{\psi^Z_1} & GL_{n}(\mathbb{Z}/p) \ar@{->}[r] & 1. }
\end{equation*}
Notice that $-Id$ is an element of the center of $GL_{n}(\mathbb{Z}/p),$ which acts on $Hom(H_p,\extp^2 H_p)$ by the multiplication of $-1.$
Then, by Center Kills Lemma,
$$H^2(GL_{n}(\mathbb{Z}/p);Hom(H_p,\extp^2 H_p))=0.$$
Therefore the extension \eqref{ext_split_1} splits.

\textbf{Case $\bullet=S$ and $p$ prime with $(p,n)\neq (3,2), (2,2), (2,3).$} 
In this case, we have that
$$ Aut\; \mathcal{N}^S_1 =Aut\; H_p\cong GL_{n}(\mathbb{Z}/p).$$
Then, the extension \eqref{ext_split_1} becomes
\begin{equation}
\label{ses_nosplit_S}
\xymatrix@C=7mm@R=10mm{0 \ar@{->}[r] & Hom(H_p,\mathcal{L}^S_{2}) \ar@{->}[r]^-{i} & Aut\;\mathcal{N}^S_{2} \ar@{->}[r]^-{\psi^S_1} & GL_{n}(\mathbb{Z}/p) \ar@{->}[r]  & 1.}
\end{equation}
Set
$\overline{\mathcal{N}^S_{2}}=\Gamma /\Gamma_2 \Gamma^{p^2}, \; \overline{\mathcal{L}^S_{2}}=\Gamma^p/\Gamma_2 \Gamma^{p^2}.$ Notice that $\Gamma_2/\Gamma_3^S$ is a characteristic subgroup of $\mathcal{N}_2^S.$ Then there is a well defined homomorphism $\widetilde{q}:Aut\;\mathcal{N}^S_{2} \rightarrow Aut\;\overline{\mathcal{N}^S_{2}}$ and there is a push-out diagram
\begin{equation}
\label{com_diag_S2}
\xymatrix@C=7mm@R=10mm{0 \ar@{->}[r] & Hom(H_p,\mathcal{L}^S_{2})\ar@{->}[d]^q \ar@{->}[r]^-{i} & Aut\;\mathcal{N}^S_{2} \ar@{->}[d]^{\widetilde{q}} \ar@{->}[r]^-{\psi^S_1} & GL_{n}(\mathbb{Z}/p) \ar@{->}[r] \ar@{=}[d] & 1 \\
0 \ar@{->}[r] & Hom(H_p,\overline{\mathcal{L}^S_{2}}) \ar@{->}[r]^-{i} & Aut\;\overline{\mathcal{N}^S_{2}} \ar@{->}[r]^-{\overline{\psi^S_1}} & GL_{n}(\mathbb{Z}/p) \ar@{->}[r] & 1, }
\end{equation}
where $q,$ $\widetilde{q}$ are induced by the quotient map respect to $\Gamma_2/\Gamma_3^S.$
Notice that $$Hom(H_p,\overline{\mathcal{L}^S_{2}})\cong \mathfrak{gl}_{n}(\mathbb{Z}/p), \qquad Aut\;\overline{\mathcal{N}^S_{2}}\cong GL_{n}(\mathbb{Z}/p^2).$$
Thus, the bottom row of diagram \eqref{com_diag_S2} becomes
\begin{equation}
\label{ses_GL_modp}
\xymatrix@C=7mm@R=10mm{ 0 \ar@{->}[r] & \mathfrak{gl}_{n}(\mathbb{Z}/p) \ar@{->}[r] &  GL_{n}(\mathbb{Z}/p^2) \ar@{->}[r]^-{r_p} &  GL_{n}(\mathbb{Z}/p) \ar@{->}[r] & 1.}
\end{equation}
Then $q_*:H^2(GL_{n}(\mathbb{Z}/p);Hom(H_p,\mathcal{L}^S_2))\rightarrow H^2(GL_{n}(\mathbb{Z}/p);\mathfrak{gl}_{n}(\mathbb{Z}/p))$ sends the cohomology class of the extension \eqref{ses_nosplit_S} to the cohomology class of the short exact sequence \eqref{ses_GL_modp}.
Since, by Corollary \eqref{cor_put_GL_split}, the extension \eqref{ses_GL_modp} does not split for $(p,n)\neq (3,2), (2,2), (3,2),$ we have that the extension \eqref{ses_nosplit_S} does not split for $(p,n)\neq (3,2),$ $(2,2), (2,3).$

\textbf{Case $\bullet=S$ with $(p,n)=(3,2),(2,2),(2,3).$} In this case we have an extension
\begin{equation}
\label{ses_split_S_exotic}
\xymatrix@C=7mm@R=10mm{0 \ar@{->}[r] & Hom(H_p,\mathcal{L}^S_{2}) \ar@{->}[r]^-{i} & Aut\;\mathcal{N}^S_{2} \ar@{->}[r]^-{\psi^S_1} & GL_{n}(\mathbb{Z}/p) \ar@{->}[r]  & 1.}
\end{equation}
To deal with these particular cases we use the following result:
\begin{teo}[Theorem III.10.3 in \cite{brown}]
\label{teo_trans}
Let $G$ be a finite group and $H$ a $p$-Sylow subgroup. For any $G$-module $M$ and any $n>0,$ the pull-back of the inclusion map $H\hookrightarrow G$ maps the $p$-primary component of $H^n(G,M)$ isomorphically onto the set of $G$-invariant elements of $H^n(H;M).$
\end{teo}
By \cite{weir} we know that the Upper triangular matrix group $UT_n(\mathbb{Z}/p)$ is a $p$-Sylow  subgroup of $GL_n(\mathbb{Z}/p).$

Consider $Aut^{\text{UT}}\mathcal{N}^S_{2}$ the preimage of $UT_n(\mathbb{Z}/p)$ by $\psi^S_1.$
We have a pull-back diagram
\begin{equation}
\xymatrix@C=7mm@R=10mm{0 \ar@{->}[r] & Hom(H_p,\mathcal{L}^S_{2}) \ar@{->}[r] \ar@{=}[d] & Aut^{\text{UT}}\mathcal{N}^S_{2} \ar@{->}[r]^-{\psi^S_1} \ar@{->}[d] & UT_n(\mathbb{Z}/p) \ar@{->}[r] \ar@{->}[d] & 1 \\
0 \ar@{->}[r] & Hom(H_p,\mathcal{L}^S_{2}) \ar@{->}[r] & Aut\;\mathcal{N}^S_{2} \ar@{->}[r]^-{\psi^S_1} & SL_n(\mathbb{Z}/p) \ar@{->}[r]  & 1.}
\end{equation}
By Theorem \ref{teo_trans} the top extension splits if and only if the bottom extension splits.

Next we show that the top extension of this pull-back diagram does not split for $(p,n)=(2,3),$ and splits for $(p,n)=(3,2),(2,2),$ giving an explicit section of $\psi^S_1: Aut^{\text{UT}}\mathcal{N}^S_{2} \rightarrow UT_n(\mathbb{Z}/p)$ for each case.

Consider $\Gamma$ the free group of rank $n$ generated by $\{ x_1,\ldots, x_n \}.$

\begin{itemize}
\item For $(p,n)=(2,2).$ Taking $T_{12}=\left(\begin{smallmatrix}
1 & 1 \\
0 & 1 
\end{smallmatrix}\right),$ the group $UT_2(\mathbb{Z}/2)$ has the following presentation:
$UT_2(\mathbb{Z}/2)=\langle T_{12} \mid T_{12}^2=1\rangle.$
Define a map $s: UT_2(\mathbb{Z}/2) \rightarrow Aut^{\text{UT}}\mathcal{N}^S_{2}$ as follows:
$$s(T_{12})=\left\{
\begin{aligned}
x_1 & \longmapsto x_1^{-1} \\
x_2 & \longmapsto x_2x_1
\end{aligned}
 \right.$$
A direct computation shows that $\psi_1^S(s(T_{12}))= T_{12}$ and $s(T_{12})^2=id.$ Therefore the map $s$ is a section of $\psi^S_1: Aut^{\text{UT}}\mathcal{N}^S_{2} \rightarrow UT_2(\mathbb{Z}/2).$

\item For $(p,n)=(3,2).$ Taking $T_{12}=\left(\begin{smallmatrix}
1 & 1 \\
0 & 1 
\end{smallmatrix}\right),$ the group $UT_2(\mathbb{Z}/3)$ has the following presentation:
$UT_2(\mathbb{Z}/3)=\langle T_{12} \mid T_{12}^3=1\rangle.$
Define a map $s: UT_2(\mathbb{Z}/3) \rightarrow Aut^{\text{UT}}\mathcal{N}^S_{2}$ as follows:
$$s(T_{12})=\left\{
\begin{aligned}
x_1 & \longmapsto x_1 x_2^6 \\
x_2 & \longmapsto x_2x_1
\end{aligned}
 \right.$$
Notice that in this case, $\Gamma^S_3=[\Gamma,\Gamma_2 \Gamma^3] (\Gamma_2 \Gamma^3)^3.$ A direct computation shows that $\psi_1^S(s(T_{12}))= T_{12}$ and $s(T_{12})^3=id.$
As a consequence, the map $s$ is a section of $\psi^S_1: Aut^{\text{UT}}\mathcal{N}^S_{2} \rightarrow UT_2(\mathbb{Z}/3).$

\item For $(p,n)=(2,3).$ Consider matrices the upper triangular matrices
$$T_{12}=\left(\begin{smallmatrix}
1 & 1 & 0 \\
0 & 1 & 0 \\
0 & 0 & 1
\end{smallmatrix}\right),\quad
T_{23}=\left(\begin{smallmatrix}
1 & 0 & 0 \\
0 & 1 & 1 \\
0 & 0 & 1
\end{smallmatrix}\right), \quad
T_{13}=\left(\begin{smallmatrix}
1 & 0 & 1 \\
0 & 1 & 0 \\
0 & 0 & 1
\end{smallmatrix}\right).$$
Then the group $UT_3(\mathbb{Z}/2)$ has the following presentation:
$$\langle T_{12}, T_{13}, T_{23} \mid T_{12}^2=T_{13}^2=T_{23}^2=[T_{12},T_{13}]=[T_{23},T_{13}]=1, [T_{12},T_{23}]=T_{13} \rangle.$$
Consider the $UT_3(\mathbb{Z}/2)$-module $M=Hom(H_2,\mathcal{L}^S_{2})$ and a central extension
$$\xymatrix@C=7mm@R=10mm{0 \ar@{->}[r] & \mathbb{Z}/2 \ar@{->}[r] & UT_3(\mathbb{Z}/2) \ar@{->}[r] & (\mathbb{Z}/2)^2 \ar@{->}[r]  & 1,}$$
where $\mathbb{Z}/2\cong \langle T_{13}\rangle,$ $(\mathbb{Z}/2)^2\cong UT_3(\mathbb{Z}/2)/\langle T_{13}\rangle.$ 
By \cite{seven} we have an associated seven term exact sequence
$$\xymatrix@C=5mm@R=5mm{0 \ar@{->}[r] & H^1((\mathbb{Z}/2)^2 ; M^{\mathbb{Z}/2}) \ar@{->}[r]^{\text{inf}} & H^1(UT_3(\mathbb{Z}/2);M) \ar@{->}[r]^{\text{res}} & H^1(\mathbb{Z}/2;M)^{(\mathbb{Z}/2)^2} \ar@{->}[r]  &  }$$
$$\xymatrix@C=5mm@R=5mm{ \ar@{->}[r] & H^2((\mathbb{Z}/2)^2; M^{\mathbb{Z}/2}) \ar@{->}[r]^-{\text{inf}} & H^2(UT_3(\mathbb{Z}/2);M)_1 \ar@{->}[r]^-{\rho} & H^1((\mathbb{Z}/2)^2; H^1(\mathbb{Z}/2;M)) \ar@{->}[r] & }$$
$$\xymatrix@C=5mm@R=5mm{ \ar@{->}[r]  & H^3((\mathbb{Z}/2)^2; M^{\mathbb{Z}/2})},$$
where $H^2(UT_3(\mathbb{Z}/2);M)_1$ is the kernel of the restriction map
$$res: H^2(UT_3(\mathbb{Z}/2);M)\longrightarrow H^2(\mathbb{Z}/2;M).$$

We show that the cohomology class associated to the extension
\begin{equation}
\label{ext_UT_(2,3)}
\xymatrix@C=5mm@R=5mm{
0 \ar@{->}[r] & Hom(H_2,\mathcal{L}^S_{2}) \ar@{->}[r] & Aut^{\text{UT}}\mathcal{N}^S_{2} \ar@{->}[r]^-{\psi^S_1} & UT_3(\mathbb{Z}/2) \ar@{->}[r] & 1 }
\end{equation}
belongs to $H^2(UT_3(\mathbb{Z}/2);M)_1$ and that the image of this cohomology class by $\rho$ is not zero getting the desired result.

Denote by $\widetilde{T_{12}},\widetilde{T_{13}}\in Aut^{UT}(\mathcal{N}_2^S)$ the following automorphims:
$$
\widetilde{T_{12}}=\left\{
\begin{array}{ll}
x_1 & \longmapsto x_1^{-1} \\
x_2 & \longmapsto x_2x_1  \\
x_3 & \longmapsto x_3
\end{array}
 \right. , \qquad
\widetilde{T_{13}}=\left\{
\begin{array}{ll}
x_1 & \longmapsto x_1^{-1} \\
x_2 & \longmapsto x_2 \\
x_3 & \longmapsto x_3 x_1
\end{array}
 \right. .$$

Consider a pull-back diagram
\begin{equation}
\label{diag_pull-back_(2,3)}
\xymatrix@C=7mm@R=10mm{
0 \ar@{->}[r] & Hom(H_2,\mathcal{L}^S_{2}) \ar@{->}[r] & Aut^{13}\;\mathcal{N}^S_{2} \ar@{->}[r]^-{\psi^S_1} & \mathbb{Z}/2 \ar@{->}[r]  & 1 \\
0 \ar@{->}[r] & Hom(H_2,\mathcal{L}^S_{2}) \ar@{->}[r] \ar@{=}[u] & Aut^{\text{UT}}\mathcal{N}^S_{2} \ar@{->}[r]^-{\psi^S_1} \ar@{<-}[u] & UT_n(\mathbb{Z}/2) \ar@{->}[r] \ar@{<-}[u] & 1 \\}
\end{equation}
where $Aut^{13}\;\mathcal{N}^S_{2}$ is the preimage of $T_{13}$ by $\psi^S_1.$

Define $s: \mathbb{Z}/2 \rightarrow Aut^{13}\mathcal{N}^S_{2}$ a map given by $s(T_{13})=\widetilde{T_{13}}.$ A direct computation shows that $\psi_1^S(s(T_{12}))= T_{12}$ and $s(T_{13})^2=1.$
Therefore the map $s$ is a section of $\psi^S_1: Aut^{13}\mathcal{N}^S_{2} \rightarrow \mathbb{Z}/2$ and hence the top extension of diagram \eqref{diag_pull-back_(2,3)} splits.
As a consequence, the cohomology class associated to the
extension \eqref{ext_UT_(2,3)} belongs to $H^2(UT_3(\mathbb{Z}/2);M)_1.$

Next we prove that the image of this cohomology class by
$$\rho: H^2(UT_3(\mathbb{Z}/2);M)_1 \longrightarrow H^1((\mathbb{Z}/2)^2; H^1(\mathbb{Z}/2;M))$$
is not zero.

By Corollary 5.10 in \cite{seven} the image of this cohomology class is zero if and only if for every $e\in Aut^{\text{UT}}\mathcal{N}^S_{2}$ there exists an element $m_e\in M$ such that
$$i(m_e)  es(\psi^S_1(e)T_{13}(\psi^S_1(e))^{-1})e^{-1} i(m_e)^{-1}=s(T_{13}).$$
Since $\langle T_{13} \rangle$ is the center of $UT_3(\mathbb{Z}/2),$
this equality becomes
\begin{equation}
\label{eq_split_condition}
i(m_e)  es(T_{13})e^{-1} i(m_e)^{-1}=s(T_{13}).
\end{equation}

In the sequel, we show that there exists an element $e\in Aut^{\text{UT}}\mathcal{N}^S_{2}$ such that for all $m_e\in M$ the equality \eqref{eq_split_condition} does not hold.
Consider $e= \widetilde{T_{12}}\in Aut^{\text{UT}}\mathcal{N}^S_{2}$
and an arbitrary element $i(m_e)\in i(Hom(H_2,\mathcal{L}^S_{2}))$ given by
$$i(m_e)=\left\{
\begin{array}{ll}
x_1 & \longmapsto x_1 a \\
x_2 & \longmapsto x_2 b \\
x_3 & \longmapsto x_3 c
\end{array}
\right. ,$$
with $a,b,c\in \mathcal{L}^S_{2}.$
Notice that $\Gamma^S_3=[\Gamma,\Gamma_2 \Gamma^2] (\Gamma_2 \Gamma^2)^2,$ the elements $\widetilde{T_{12}},$ $i(m_e)$ have order $2,$ the action of $Aut^{UT}(\mathcal{N}_2^S)$ on $\mathcal{L}_2^S$ factors through $UT_3(\mathbb{Z}/2)$ and $T_{13}$ is central in $UT_3(\mathbb{Z}/2).$
Then we have that
$$i(m_e)  es(T_{13})e^{-1} i(m_e)^{-1}=i(m_e) \widetilde{T_{12}}\widetilde{T_{13}}\widetilde{T_{12}} i(m_e)=$$
$$=\left\{
\begin{array}{lllll}
x_1 & \longmapsto x_1 a & \longmapsto x_1^{-1} T_{12}(a) & \longmapsto x_1 T_{13}T_{12}(a) & \longmapsto \\
x_2 & \longmapsto x_2 b & \longmapsto x_2x_1 T_{12}(b) & \longmapsto x_2x_1^{-1} T_{13}T_{12}(b) & \longmapsto \\
x_3 & \longmapsto x_3 c & \longmapsto x_3 T_{12}(c) & \longmapsto x_3x_1 T_{13}T_{12}(c) & \longmapsto
\end{array}
\right. $$

$$\begin{array}{ll}
 \longmapsto x_1^{-1} T_{12}T_{13}T_{12}(a)\;=\; x_1^{-1} T_{13}(a) & \longmapsto x_1^{-1} a T_{13}(a) \\
 \longmapsto x_2x_1^2 T_{12}T_{13}T_{12}(b)\;=\; x_2x_1^2 T_{13}(b) & \longmapsto x_2x_1^2 b T_{13}(b) \\
 \longmapsto x_3x_1^{-1} T_{12}T_{13}T_{12}(c)\;=\; x_3x_1x_1^2 T_{13}(c) & \longmapsto x_3x_1x_1^2 ac T_{13}(c).
\end{array}$$
Hence, to get the equality \eqref{eq_split_condition} it is necessary that $x_1^2=bT_{13}(b)$ for some $b\in \mathcal{L}_2^S.$
Notice that $\mathcal{L}_2^S$ is a $\mathbb{Z}/2$-vector space generated by
$$\langle x_1^2,x_2^2,x_3^2, [x_1,x_2], [x_1,x_3], [x_2,x_3] \rangle.$$
Moreover we have that
$$x_1^2T_{13}(x_1^2)=x_2^2T_{13}(x_2^2)=[x_1,x_2]T_{13}([x_1,x_2])=
[x_1,x_3]T_{13}([x_1,x_3])=1,$$
$$x_3^2T_{13}(x_3^2)=x_3^2(x_3x_1)^2=x_1^2[x_1,x_3],\quad [x_2,x_3]T_{13}([x_2,x_3])=[x_2,x_1].$$
As a consequence $x_1^2$ is not a linear combination of elements $wT_{13}(w)$ with $w$ a generator element of $\mathcal{L}_2^S$
and then there does not exist $b\in \mathcal{L}_2^S$ such that $bT_{13}(b).$
\end{itemize}

\end{proof}

Next we show that for $k\geq 2,$ the extensions of the beginning of this Section do not split.
In order to deal with this case, we first prove that, for a free group of finite rank, Lemma \eqref{lema_coop_general} and Corollary \eqref{cor_IA_bullet} give us the most efficient bound, in the sense that
$$\begin{aligned}
& [IA_k^\bullet(\Gamma),\Gamma^\bullet_l]<\Gamma^\bullet_{k+l}\\
& [IA^\bullet_k(\Gamma),IA^\bullet_l(\Gamma)] < IA^\bullet_{k+l}(\Gamma)
\end{aligned}
\qquad\text{but}\qquad
\begin{aligned}
& [IA_k^\bullet(\Gamma),\Gamma^\bullet_l]\nless\Gamma^\bullet_{k+l+1} \\
& [IA^\bullet_k(\Gamma),IA^\bullet_l(\Gamma)] \nless IA^\bullet_{k+l+1}(\Gamma).
\end{aligned}$$

\begin{prop}
\label{prop_cotas_IA_bullet,N_bullet}
Let $\Gamma$ be a free group of finite rank $n > 1.$ Then
\begin{enumerate}[i)]
\item $[IA_k^\bullet(\Gamma),\Gamma^\bullet_l]\nless\Gamma^\bullet_{k+l+1},$
\item $[IA^\bullet_k(\Gamma),IA^\bullet_l(\Gamma)] \nless IA^\bullet_{k+l+1}(\Gamma).$
\end{enumerate}
\end{prop}

\begin{proof}
Consider the free group $\Gamma=\langle x_1,\ldots, x_n \rangle.$ Let $\gamma_1=[x_1,[x_2,[x_1,[x_2,\ldots]]]],$ a commutator of length $k$ alternating $x_1,$ $x_2,$ and $\gamma_2=[x_1,[x_2,[x_1,[x_2,\ldots]]]]$ a commutator of length $l$ alternating $x_1,$ $x_2.$

\textbf{i)} Take $f\in IA_k^\bullet(\Gamma)$ the inner automorphism given by the conjugation by $\gamma_1.$ Notice that $f\in IA_k^\bullet(\Gamma)$ because for every $x\in \Gamma,$
$$f(x)=\gamma_1x\gamma_1^{-1}= [\gamma_1,x]x=x \;(\text{mod } \Gamma_{k+1}).$$
Then we have that $[f,\gamma_2]\in [IA^\bullet_k(\Gamma),\Gamma^\bullet_l]$ and
$$[f,\gamma_2]=f(\gamma_2)\gamma_2^{-1}=\gamma_1\gamma_2\gamma_1^{-1}\gamma_2^{-1}=[\gamma_1,\gamma_2] \; \in \Gamma_{k+l}.$$
Notice that $[\gamma_1,\gamma_2]$ has weight $k+l.$ By properties of free Lie algebras (see \cite{hall3}), this commutator can only be annihilated by commutators of the same weight.
Thus, if $[\gamma_1,\gamma_2]\in \Gamma_{k+l+1}^\bullet,$ by the concrete description of the layers of our $p$-central series \eqref{zassenhaus_series_def}, \eqref{stallings_series_alternative_def}, we would have that $[\gamma_1,\gamma_2]\in (\Gamma_{k+l})^p.$ But, by construction $[\gamma_1,\gamma_2]$ is not a product of $p$-powers of commutators of weight $k+l.$
Therefore, $[\gamma_1,\gamma_2]$ does not belong to $\Gamma_{k+l+1}^\bullet.$

\textbf{ii)} Take $\varphi\in IA_k^\bullet(\Gamma),$ $\psi \in IA_l^\bullet(\Gamma)$ the inner automorphisms respectively given by the conjugation by $\gamma_1$ and $\gamma_2.$
Then we have that $[\varphi, \psi]\in [IA_k^\bullet(\Gamma),IA_l^\bullet(\Gamma)]$ and
$$
[\varphi, \psi](x_1)x_1^{-1}=[\gamma_1,\gamma_2]x_1[\gamma_1,\gamma_2]^{-1}x_1^{-1}=[[\gamma_1,\gamma_2],x_1]\;\in \Gamma_{k+l+1},
$$
which, as in $i),$ does not belong to $\Gamma_{k+l+2}^\bullet.$
Therefore, $[\phi,\psi]\notin IA^\bullet_{k+l+1}(\Gamma).$
\end{proof}

As a direct consequence we have the following result:

\begin{cor}
\label{coro_cota_IA_bullet,N_bullet}
The group $[IA^\bullet_k(\mathcal{N}_n^\bullet),(\mathcal{N}_n^\bullet)^\bullet_l]$ with $l+k=n$ and $[IA_k^\bullet(\mathcal{N}^\bullet_{n}),IA_l^\bullet(\mathcal{N}^\bullet_{n})]$ with $l+k=n-1,$ are not trivial.
\end{cor}

\begin{proof}
We use the same notation that we used in proof of Proposition \eqref{prop_cotas_IA_bullet,N_bullet}.
Consider the elements $f\in IA_k^\bullet(\Gamma)$ and $x\in \Gamma^\bullet_l$ as in the proof of Proposition \eqref{prop_cotas_IA_bullet,N_bullet}. These elements induce elements $\overline{f}\in IA_k^\bullet(\Gamma)$ and $\overline{x}\in \Gamma^\bullet_l.$ Moreover, by Proposition \eqref{prop_cotas_IA_bullet,N_bullet}, we have that
$$[\overline{f},\overline{x}]=\overline{[f,x]}=\overline{[\gamma_1,\gamma_2]}\neq 1.$$
Therefore there is an element in $[IA^\bullet_k(\mathcal{N}_n^\bullet),(\mathcal{N}_n^\bullet)^\bullet_l]$ which is not trivial.

Take $\varphi\in IA_k^\bullet(\Gamma),$ $\psi\in IA_l^\bullet(\Gamma),$ and $x_1\in \Gamma.$ These elements induce elements $\overline{\varphi}\in IA_k^\bullet(\mathcal{N}_n^\bullet),$ $\overline{\psi}\in IA_l^\bullet(\mathcal{N}_n^\bullet),$ and $\overline{x_1}\in \mathcal{N}_n^\bullet.$
Moreover, by Proposition \eqref{prop_cotas_IA_bullet,N_bullet}, we have that
$$[[\overline{\varphi},\overline{\psi}],\overline{x_1}]=\overline{[[\varphi,\psi],x_1]}=\overline{[[\gamma_1,\gamma_2],x_1]}\neq 1.$$
Therefore there is an element in $[IA_k^\bullet(\mathcal{N}^\bullet_{n}),IA_l^\bullet(\mathcal{N}^\bullet_{n})]$ which is not trivial.
\end{proof}

\begin{prop}
\label{prop_non-split_versal}
Let $\Gamma$ be a free group of finite rank $n > 1.$ The extension
\begin{equation*}
\xymatrix@C=7mm@R=10mm{0 \ar@{->}[r] & Hom(\mathcal{N}^\bullet_1,\mathcal{L}^\bullet_{k+1}) \ar@{->}[r]^-{i} & IA^p(\mathcal{N}^\bullet_{k+1}) \ar@{->}[r]^-{\psi^\bullet_k} & IA^p(\mathcal{N}^\bullet_k ) \ar@{->}[r] & 1, }
\end{equation*}
does not split.
\end{prop}

\begin{proof}
Consider the central extension
\begin{equation}
\label{ext_split_k}
\xymatrix@C=7mm@R=10mm{0 \ar@{->}[r] & Hom(\mathcal{N}^\bullet_1,\mathcal{L}^\bullet_{k+1}) \ar@{->}[r]^-{i} & IA^p(\mathcal{N}^\bullet_{k+1}) \ar@{->}[r]^-{\psi^\bullet_k} & IA^p(\mathcal{N}^\bullet_k ) \ar@{->}[r] & 1. }
\end{equation}
The associated 5-term sequence give us an exact sequence
\begin{equation}
\label{les_5-term}
\begin{gathered}
\xymatrix@C=7mm@R=10mm{ Hom(IA^p(\mathcal{N}^\bullet_{k+1});Hom(\mathcal{N}^\bullet_1,\mathcal{L}^\bullet_{k+1})) \ar@{->}[r]^-{res} & Hom(Hom(\mathcal{N}^\bullet_1,\mathcal{L}^\bullet_{k+1}),Hom(\mathcal{N}^\bullet_1,\mathcal{L}^\bullet_{k+1}))\ar@{->}[r] &}
\\
\xymatrix@C=7mm@R=10mm{ \ar@{->}[r]^-{\delta} & H^2(IA^p(\mathcal{N}^\bullet_k ); Hom(\mathcal{N}^\bullet_1,\mathcal{L}^\bullet_{k+1})).}
\end{gathered}
\end{equation}
Then the cohomology class associated to the extension \eqref{ext_split_k} is given by $\delta(id).$

Suppose that the central extension \eqref{ext_split_k} splits, i.e. $\delta(id)=0.$ Since \eqref{les_5-term} is exact, there would be an element $f\in Hom(IA^p(\mathcal{N}^\bullet_{k+1});Hom(\mathcal{N}^\bullet_1,\mathcal{L}^\bullet_{k+1}))$ such that $res(f)=id.$

Notice that for every element $x\in [IA^p(\mathcal{N}^\bullet_{k+1}),IA^p(\mathcal{N}^\bullet_{k+1})],$ one has that $f(x)=0$ because $Hom(\mathcal{N}^\bullet_1,\mathcal{L}^\bullet_{k+1})$ is abelian.
On the other hand, by Corollary \eqref{coro_cota_IA_bullet,N_bullet},
$$[IA^p(\mathcal{N}^\bullet_{k+1}),IA^\bullet_{k-1}(\mathcal{N}^\bullet_{k+1})]\neq 1.$$
In addition, by Corollary \eqref{cor_IA_bullet},
$$\psi^\bullet_k([IA^p(\mathcal{N}^\bullet_{k+1}),IA^\bullet_{k-1}(\mathcal{N}^\bullet_{k+1})])=[IA^p(\mathcal{N}^\bullet_{k}),IA^\bullet_{k-1}(\mathcal{N}^\bullet_{k})]\leq IA^\bullet_{k}(\mathcal{N}^\bullet_{k})=1.$$
Then, by short exact sequence \eqref{ses_ZS},
$[IA^p(\mathcal{N}^\bullet_{k+1}),IA^\bullet_{k-1}(\mathcal{N}^\bullet_{k+1})] \hookrightarrow Hom(\mathcal{N}^\bullet_1,\mathcal{L}^\bullet_{k+1}).$

As a consequence, $res(f)$ can not be the identity.
\end{proof}

\begin{rem}
We point out that using the same argument for the lower central series instead the mop $p$ central series one gets that $[IA_k(\Gamma),IA_m(\Gamma)]<IA_{k+m}(\Gamma)$ is the most efficient bound in the sense that $[IA_k(\Gamma),IA_l(\Gamma)] \nless IA_{k+l+1}(\Gamma),$ and as a consequence, the central extension
\begin{equation*}
\xymatrix@C=7mm@R=10mm{0 \ar@{->}[r] & Hom(\mathcal{N}_1,\mathcal{L}_{k+1}) \ar@{->}[r]^-{i} & IA(\mathcal{N}_{k+1}) \ar@{->}[r]^-{\psi_k} & IA(\mathcal{N}_k ) \ar@{->}[r] & 1 }
\end{equation*}
does not split.
\end{rem}

Considering a pull-back diagram
\begin{equation*}
\xymatrix@C=7mm@R=10mm{0 \ar@{->}[r] & Hom(\mathcal{N}^\bullet_1,\mathcal{L}^\bullet_{k+1})\ar@{=}[d] \ar@{->}[r]^-{i} & IA^p(\mathcal{N}^\bullet_{k+1})\ar@{->}[r]^-{\psi^\bullet_k} \ar@{^{(}->}[d] & IA^p(\mathcal{N}^\bullet_k) \ar@{^{(}->}[d] \ar@{->}[r] & 1\\
0 \ar@{->}[r] & Hom(\mathcal{N}^\bullet_1,\mathcal{L}^\bullet_{k+1}) \ar@{->}[r]^-{i} & Aut\;\mathcal{N}^\bullet_{k+1} \ar@{->}[r]^-{\psi^\bullet_k} & Aut \;\mathcal{N}^\bullet_k \ar@{->}[r] & 1
,}
\end{equation*}
we have that the top extension of this diagram splits if the bottom extension of this diagram splits (c.f. Section IV.3 of \cite{brown}). Therefore,
as a direct consequence of Proposition \ref{prop_non-split_versal}, we get the expected result:

\begin{cor}
Let $\Gamma$ be a free group of finite rank $n > 1.$ The extension
\begin{equation*}
\xymatrix@C=7mm@R=10mm{0 \ar@{->}[r] & Hom(\mathcal{N}^\bullet_1,\mathcal{L}^\bullet_{k+1}) \ar@{->}[r]^-{i} & Aut\;\mathcal{N}^\bullet_{k+1} \ar@{->}[r]^-{\psi^\bullet_k} & Aut \;\mathcal{N}^\bullet_k \ar@{->}[r] & 1 }
\end{equation*}
does not split for $k\geq 2.$
\end{cor}

\vspace{1cm}
\textbf{Acknowledgements}

The author would like to express his gratitude to the Professor Wolfgang Pitsch for his encouragement and guidance throughout all this work.


\newpage
\bibliography{biblio}{}
\bibliographystyle{abbrv}

\end{document}